\documentclass[10pt]{article}

\usepackage{amsfonts}
\usepackage{shuffle,yfonts,amsmath,amssymb,amsthm,etex,tikz}

\newcommand{\MZV}{\mathsf{MZV}}
\newcommand{\CMZV}{\mathsf{CMZV}}
\newcommand{\MtV}{\mathsf{MtV}}
\newcommand{\AMtV}{\mathsf{AMtV}}

\newcommand{\sha}{\shuffle}
\newcommand{\si}{\sigma}

\newcommand{\ola}{\overleftarrow}
\newcommand{\ora}{\overrightarrow}

\newcommand\ta{{\texttt{a}}}
\newcommand\ty{{\texttt{y}}}
\newcommand\tx{{\texttt{x}}}
\newcommand\tz{{\texttt{z}}}

\newcommand\gl{{\lambda}}

\newcommand\ga{{\alpha}}

\newcommand\om{{\omega}}

\newcommand\eps{{\varepsilon}}

\newcommand{\bfga}{{\boldsymbol{\sl{\alpha}}}}

\newcommand{\bfk}{{\boldsymbol{\sl{k}}}}
\newcommand{\bfl}{{\boldsymbol{\sl{l}}}}
\newcommand{\bfm}{{\boldsymbol{\sl{m}}}}

\newcommand{\bfs}{{\boldsymbol{\sl{s}}}}
\newcommand{\bfr}{{\boldsymbol{\sl{r}}}}

\newcommand{\bfz}{{\boldsymbol{\sl{z}}}}

\newcommand\bfeps{{\boldsymbol \varepsilon}}

 \allowdisplaybreaks
 \allowdisplaybreaks

\catcode`!=11
\let\!int\int \def\int{\displaystyle\!int}
\let\!lim\lim \def\lim{\displaystyle\!lim}
\let\!sum\sum \def\sum{\displaystyle\!sum}
\let\!sup\sup \def\sup{\displaystyle\!sup}
\let\!inf\inf \def\inf{\displaystyle\!inf}
\let\!cap\cap \def\cap{\displaystyle\!cap}
\let\!max\max \def\max{\displaystyle\!max}
\let\!min\min \def\min{\displaystyle\!min}
\let\!frac\frac \def\frac{\displaystyle\!frac}
\catcode`!=12

\let\oldsection\section
\renewcommand\section{\setcounter{equation}{0}\oldsection}

\allowdisplaybreaks

\DeclareMathOperator*{\dep}{dep}
\DeclareMathOperator{\Li}{Li}

\def\N{\mathbb{N}}
\def\Z{\mathbb{Z}}
\def\Q{\mathbb{Q}}

\def\ze{\zeta}

\theoremstyle{plain}
\newtheorem{thm}{Theorem}[section]
\newtheorem{lem}[thm]{Lemma}
\newtheorem{cor}[thm]{Corollary}

\newtheorem{pro}[thm]{Proposition}

\theoremstyle{definition}

\newtheorem{re}[thm]{Remark}
\newtheorem{exa}[thm]{Example}
\newtheorem{qu}[thm]{Question}

\newcommand{\binn}{{\binom{2n}{n}}}

\begin{document}
\title{\bf Ap\'{e}ry-Type Series and Colored Multiple Zeta Values}
\author{
{Ce Xu${}^{a,}$\thanks{Email: cexu2020@ahnu.edu.cn, ORCID 0000-0002-0059-7420.}\ and Jianqiang Zhao${}^{b,}$\thanks{Email: zhaoj@ihes.fr, ORCID 0000-0003-1407-4230.}}\\[1mm]
\small a. School of Mathematics and Statistics, Anhui Normal University, Wuhu 241002, P.R.C\\
\small b. Department of Mathematics, The Bishop's School, La Jolla, CA 92037, United States of America}

\date{}
\maketitle

\noindent{\bf Abstract.} In this paper, we study new classes of Ap\'{e}ry-type series involving the central binomial coefficients and the multiple $t$-harmonic sums by combining the methods of iterated integrals and Fourier--Legendre series expansions, where the multiple $t$-harmonic sums are a variation of multiple harmonic sums in which all the summation indices are restricted to odd numbers only. Our approach also enables us to generalize some old classes of Ap\'{e}ry-type series involving harmonic sums to those with products of multiple harmonic sums and multiple $t$-harmonic sums. We show that these series can be expressed as either the real or the imaginary part of a $\Q$-linear combination of colored multiple zeta values of level 4. Hopefully, these relations will shed some new lights on their properties which may lead to novel approaches to irrationality questions on their properties and may lead to new approaches to irrationality questions on the Riemmann zeta values, or more generally, the multiple zeta values.



\section{Introduction}
The study of infinite series involving central binomial coefficients was first brought to he attention to the math world by
Ap\'ery with his celebrated proof of irrationality of $\zeta(2)$ and $\zeta(3)$. These series are called Ap\'{e}ry-type series or Ap\'{e}ry-like sums. Although many new identities have been found they lead to no more irrationality proofs for other odd Riemann zeta values.

However, in recent years a large number of research work have been focus on the study of multiple zeta values (MZVs) and their generalizations. Many surprising connections to other objects in mathematics and physics have been discovered. In this paper, we will present many new families of relations between Ap\'{e}ry-type series and the colored multiple zeta values, among which the Riemann zeta values are the special cases. Hopefully, these will shed some new lights to the properties of these series and may lead to new approaches to irrationality questions.

We begin with some basic notations. Let $\N$ be the set of positive integers and $\N_0:=\N\cup \{0\}$.
A finite sequence $\bfk:=(k_1,\ldots, k_r)\in\N^r$ is called a \emph{composition}. We put
\begin{equation*}
 |\bfk|:=k_1+\cdots+k_r,\quad \dep(\bfk):=r,
\end{equation*}
and call them the weight and the depth of $\bfk$, respectively. If $k_1>1$, $\bfk$ is called \emph{admissible}.

For a composition $\bfk=(k_1,\ldots,k_r)$ and a positive integer $n$, the \emph{multiple harmonic sum} and \emph{multiple harmonic star sum} are defined by
\begin{align}
\ze_n(\bfk):=\sum\limits_{n\geq n_1>\cdots>n_r>0 } \frac{1}{n_1^{k_1}\cdots n_r^{k_r}}\quad
\text{and}\quad
\ze^\star_n(\bfk):=\sum\limits_{n\geq n_1\geq\cdots\geq n_r>0} \frac{1}{n_1^{k_1}\cdots n_r^{k_r}}\label{MHSs+MHSSs},
\end{align}
respectively. If $n<k$ then ${\ze_n}(\bfk):=0$ and ${\ze_n}(\emptyset )={\ze^\star _n}(\emptyset ):=1$. When taking the limit $n\rightarrow \infty$ in \eqref{MHSs+MHSSs} we get the \emph{multiple zeta values} (MZVs) and the \emph{multiple zeta star values} (MZSVs), respectively:
\begin{align*}
{\ze}( \bfk):=\lim_{n\rightarrow \infty}{\ze_n}(\bfk), \quad
\text{and}\quad
{\ze^\star}( \bfk):=\lim_{n\rightarrow \infty}{\ze^\star_n}( \bfk),
\end{align*}
defined for admissible compositions  $\bfk$ to ensure convergence of the series. The systematic study of MZVs began in the early 1990s with the works of Hoffman \cite{H1992} and Zagier \cite{DZ1994}. Due to their surprising and sometimes mysterious appearance in the study of many branches of mathematics and theoretical physics, these special values have attracted a lot of attention and interest in the past three decades (for example, see the book by the
second author  \cite{Zhao2016}).

In general, let $\bfk=(k_1,\ldots,k_r)\in\N^r$ and $\bfz=(z_1,\dotsc,z_r)$ where $z_1,\dotsc,z_r$ are $N$th roots of unity. We can define the \emph{colored MZVs} (CMZVs) of level $N$ as
\begin{equation}\label{equ:defnMPL}
\Li_{\bfk}(\bfz):=\sum_{n_1>\cdots>n_r>0}
\frac{z_1^{n_1}\dots z_r^{n_r}}{n_1^{k_1} \dots n_r^{k_r}},
\end{equation}
which converges if $(k_1,z_1)\ne (1,1)$ (see \cite{Racinet2002} and \cite[Ch. 15]{Zhao2016}), in which case we call $({\bfk};\bfz)$ \emph{admissible}. The level two colored MZVs are often called Euler sums or alternating MZVs. In this case, namely, when $(z_1,\dotsc,z_r)\in\{\pm 1\}^r$ and $(k_1,z_1)\ne (1,1)$, we set
$\ze(\bfk;\bfz)= \Li_\bfk(\bfz)$. Further, we put a bar on top of
$k_{j}$ if $z_{j}=-1$. For example,
\begin{equation*}
\ze(\bar2,3,\bar1,4)=\ze(2,3,1,4;-1,1,-1,1).
\end{equation*}
More generally, for any composition $(k_1,\dotsc,k_r)\in\N^r$, the \emph{classical multiple polylogarithm function} with $r$-variable is defined by
\begin{align}\label{equ:classicalLi}
\Li_{k_1,\dotsc,k_r}(x_1,\dotsc,x_r):=\sum_{n_1>n_2>\cdots>n_r>0} \frac{x_1^{n_1}\dotsm x_r^{n_r}}{n_1^{k_1}\dotsm n_r^{k_r}}
\end{align}
which converges if $|x_1\cdots x_j|<1$ for all $j=1,\dotsc,r$. They can be analytically continued to a multi-valued meromorphic function on $\mathbb{C}^r$ (see \cite{Zhao2007d}). In particular, if $x_1=x,x_2=\cdots=x_r=1$, then $\Li_{k_1,\ldots,k_r}(x,1_{r-1})$ is the classical single-variable multiple polylogarithm function.

Similar to the multiple harmonic sums and the multiple harmonic star sums, for a composition $\bfk=(k_1,\ldots,k_r)$ and a positive integer $n$, we define the \emph{multiple t-harmonic sum} and \emph{multiple t-harmonic star sum} respectively by
\begin{align*}
t_n(\bfk):=\sum_{n\geq n_1>n_2>\cdots>n_r>0} \prod_{j=1}^r \frac{1}{(2n_{j}-1)^{k_{j}}}\quad
\text{and}\quad
t^\star_n(\bfk):=\sum_{n\geq n_1\geq n_2\geq \cdots\geq n_r>0} \prod_{j=1}^r \frac{1}{(2n_{j}-1)^{k_{j}}}.
\end{align*}
We set ${t_n}(\emptyset)={t^\star_n}(\emptyset ):=1$ and ${t_n}(\bfk):=0$ if $n<k$. When taking the limit $n\rightarrow \infty$ we get the so-called the \emph{multiple $t$-values} (MtVs) and the \emph{multiple $t$-star values} (MtSVs), respectively. These values has been studied first by Hoffman who introduced them as the  odd variants of MZVs and MZSVs in \cite{H2019}, respectively. In the above definition of MtVs, we put a bar on the top of $k_j$ for $j=1,2,\ldots r$ if there is a sign $(-1)^{n_j}$ appearing in the numerator on the right. Or we can also use similar notation $t(\bfs;\bfeps)$ where $\bfeps$ is a composition of $\pm 1$'s. Those involving one or more of the $k_j$ barred are called the \emph{alternating multiple $t$-values}. For example,
\begin{equation*}
t(\bar{k}_1,\bar{k}_2,k_3,k_4)=\sum_{n_1>n_2>n_3>n_4\geq1}
    \frac{(-1)^{n_1+n_2}}{(2n_1-1)^{k_1}(2n_2-1)^{k_2}(2n_3-1)^{k_3}(2n_4-1)^{k_4}}\,.
\end{equation*}
Obviously, the alternating MtVs can be expressed in terms of CMZVs of level 4.

We point out that we have defined more generally the \emph{multiple mixed values} for an admissible composition $\bfk=(k_1,\ldots,k_r)$ and $\bfeps=(\varepsilon_1,\ldots,\varepsilon_r)\in\{\pm1\}^r$ (see \cite{XuZhao2020a}) by
\begin{align}\label{defn-mmvs}
M(\bfk;\bfeps)&:=\sum_{n_1>n_2>\cdots>n_r>0} \prod_{j=1}^r \frac{1+\varepsilon_{j}(-1)^{n_{j}}}{n_{j}^{k_{j}}}.
\end{align}
These can be regarded as variants of the multiple zeta value of level two.
Also, as special cases of multiple mixed values, the first author \cite{Xu2021} defined the values
\begin{align}\label{Def-MRV}
R(k_1,k_2,\ldots,k_r):=2^{k_1+\cdots+k_r} \sum_{n_1>\cdots>n_r>0} \frac{1}{(2n_1-1)^{k_1} (2n_{2})^{k_2}\cdots(2n_r)^{k_r}},
\end{align}
for positive integers $k_1,k_2,\ldots,k_r$ with $k_1\geq 2$, which are called \emph{multiple $R$-values}. If $r=1$ and $k_1=k\geq 2$, we have
$$R(k)=2^{k}\sum_{n=1}^\infty \frac{1}{(2n-1)^k}=(2^k-1)\ze(k)\quad (k\geq 2).$$
In fact, the first author showed that (see \cite[Thm. 4.1 and Cor. 4.2]{Xu2021})
\begin{align}\label{Eq-MRV-H1}
{R}(m+1,1_{n-1})\in\Q[\log 2,\ze(2),\ze(3),\ze(4),\ldots],
\end{align}
where $1_d$ is the sequence of 1's with $d$ repetitions.

The primary goal of this paper is to study the explicit relations of Ap\'{e}ry-type series and colored multiple zeta values by using the method of iterated integrals and Fourier--Legendre series expansions. This is partially motivated by the recent work of Campbell and his collaborators, see \cite{Camp18,Camp-Aur-Son19,Can-Aur2019}.
For example, in Thm.~\ref{thm-MR-Apery-AMtVs} we will express the Ap\'{e}ry-type series involving central binomial coefficients
\begin{equation*}
\sum\limits_{n=0}^\infty \bigg[\frac1{4^n}\binn\bigg]^2 \frac1{(2n+1)^{m+1}},\quad  m\in \N_0,
\end{equation*}
explicitly in terms of alternating multiple $t$-values by computing the Fourier--Legendre series expansions of $\log(x)/\sqrt{x}$. As two general results, we show in Cor.~\ref{cor-1stPower} that
for any positive integers $p$, $m$ and $\bfk\in\N^r$, the series involving multiple $t$-harmonic sums
\begin{align*}
\sum_{n=1}^\infty \frac1{4^n}\binn\frac{t_n(\bfk)}{(2n+1)^{m}} \quad
\text{(resp.}\ \sum_{n=1}^\infty  \frac{4^{n}}{\binn} \frac{t_n(\bfk)}{n^{m+1}}), \quad  \ m\in\N,
\end{align*}
can be expressed as the imaginary (resp. real) part of a $\Q$-linear combination of CMZVs of level 4 and weight $|\bfk|+m$ (resp. $|\bfk|+m+1$).
Furthermore, in Thm.~\ref{thm-Apery-MRV} we prove that the Ap\'{e}ry-type series
\begin{equation*}
 \sum\limits_{n=1}^\infty \frac{1}{4^n}\binn \frac{t_n(1_k)}{n^{m+1}},\quad m,k\in\N_0,
\end{equation*}
can be expressed in terms of rational linear combinations of products of $\log 2$ and the Riemann zeta values.
Some related results may be found in \cite{Au2020,Camp18,ChenKW19,KWY2007,X2020} and references therein.

The remainder of this paper is organized as follows. After stating three preliminary lemmas in section \ref{sec-lemmas} we shall consider the Fourier--Legendre series expansion of $\log^m(x)$ (resp. $\log^m(x)/\sqrt{x}$) and obtain an expansion of their higher derivatives in terms the multiple harmonic sums in section \ref{sec-logm} (resp. section \ref{sec-logm-sqrt}). The result of section \ref{sec-logm} will then be applied in section \ref{sec-Apery} to show that some Ap\'{e}ry-type series can be expressed in terms of CMZVs (see Thm.~\ref{thm-ASASs1}) and Prop.~\ref{pro-inCMZV24}). Here the Fourier--Legendre series expansion of the complete elliptic integral of the first kind $K(x)$, defined by \eqref{KEQ1}, plays a crucial role. Next, in a similar vein, the result of section \ref{sec-logm-sqrt} will be applied in sections \ref{sec-Apery-MtV} and \ref{sec-Apery-MtSV} to show that some MtV and MtSV variants of Ap\'{e}ry-type series can be expressed in terms of CMZVs (see Thm.~\ref{thm-Apery-MRV} and Thm.~\ref{thm-Apery-MtV-Mhs-CB}). In sections \ref{sec-SpecialMZSV-MtVproductI} and \ref{sec-SpecialMZSV-MtVproductII}, using the theory of iterated integrals we prove a few results
for more general forms of variation of Ap\`ery-type series involving products of multiple $t$-harmonic sums and multiple harmonic sums, and relate them to CMZVs of level 2 and level 4. In particular, we consider series in which the central binomial coefficients appear both as numerator and denominator. We conclude this paper with some further questions for future research along this direction in the last section. Our numerical computation throughout the paper was done using MAPLE.

\section{Some preliminary lemmas}\label{sec-lemmas}
In this section, we collect three lemmas which will be used throughout the rest of the paper.

\begin{lem}\label{lem1} \emph{(\cite{C1974})} Fa$\grave{a}$ di Bruno's formula may be stated in terms of Bell polynomials as follows:
\begin{align}\label{2.3}
\frac {d^n}{dx^n}f(g(x))=\sum\limits_{k=1}^n f^{(k)}(g(x))B_{n,k}(g^{(1)}(x),g^{(2)}(x),\cdots,g^{(n-k+1)}(x)),\quad n\in\N,
\end{align}
where $B_{n,k}$ is the exponential partial Bell polynomials defined by
\begin{align}\label{equ:Bell}
\frac{1}{k!} \left(\sum\limits_{j=1}^\infty x_j \frac{t^j}{j!} \right)^k=\sum\limits_{n=k}^\infty B_{n,k}(x_1,\ldots,x_{n-k+1}) \frac{t^n}{n!},\quad k=0,1,2,\ldots.
\end{align}
Moreover, we have the following recurrence relation
\begin{align}\label{3}
B_{n,k}(x_1,\ldots,x_{n-k+1}) =\sum\limits_{i=1}^{n-k+1} \binom{n-1}{i-1}x_i B_{n-i,k-1}(x_1,\ldots,x_{n-k-i+2}) ,
\end{align}
with $B_{0,0}(x_1)=1,\ B_{n,0}(x_1,\ldots,x_{n+1})=0\ (n\geq 1),\ B_{0,k}(\cdot)=0\ (k\geq 1)$.
\end{lem}

\begin{lem}\label{lem-S-1}   \emph{(\cite[Thm. 4.1]{Xu2017})}
Define two sequences ${A_m(n)}$ and ${B_m(n)}$ by
$$ A_m(n) = (m-1)!\underset{i = 0}{\overset{m - 1}{\sum}}\dfrac{A_i(n)}{i!}\underset{k = 1}{\overset{n}{\sum}}x_k^{m - i},\quad A_0(n) = 1,\quad \left( {{x_k} \in \mathbb{C},k = 1,2, \cdots ,n} \right),$$
$$ B_m(n) = \underset{k_1 = 1}{\overset{n}{\sum}}x_{k_1}\underset{k_2 = 1}{\overset{k_1}{\sum}}x_{k_2}\cdots\underset{k_m = 1}{\overset{k_{m-1}}{\sum}}x_{k_m},\quad B_0(n) = 1, \quad \left( {{x_k} \in \mathbb{C},k = 1,2, \cdots ,n} \right).$$
Then \[A_m(n) = m!B_m(n).\]
\end{lem}

\begin{lem}\label{lem-S-2} \emph{(\cite[Thm. 4.2]{Xu2017})}
Define two sequences ${{\bar A}_m(n)}$ and ${{\bar B}_m(n)}$ by
$${\bar A}_m(n) = (m-1)!(-1)^{m-1}\underset{i = 0}{\overset{m - 1}{\sum}}(-1)^{i}\dfrac{{\bar A}_i(n)}{i!}\underset{k = 1}{\overset{n}{\sum}}x_k^{m - i},\quad {\bar A}_0(n)=1,$$
$${\bar B}_m(n) = \underset{k_1 = 1}{\overset{n}{\sum}}x_{k_1}\underset{k_2 = 1}{\overset{k_1-1}{\sum}}x_{k_2}\cdots\underset{k_m = 1}{\overset{k_{m-1}-1}{\sum}}x_{k_m},\quad {\bar B}_0(n) = 1.$$
Then \[{\bar A}_m(n) = m!{\bar B}_m(n).\]
\end{lem}

\section{Expansion of $\log^m(x)$ and their higher derivatives}\label{sec-logm}
In this section, we will apply Legendre polynomials to obtain some expansions of $\log^m(x)$ and its higher derivatives. Recall that the
Rodrigues formula for the Legendre polynomials has the form
\begin{equation}\label{equ:LegendrePoly}
     P_n(x)=\frac{1}{2^n n!} \frac{d^n}{dx^n} (x^2-1)^n.
\end{equation}
Then we can find easily that
\begin{equation}\label{equ:PnAltDefn}
P_{n}(2x-1)=\frac{1}{n!}\frac{d^n}{dx^n}[x^n(x-1)^n].
\end{equation}

\begin{thm}\label{thm-LP-Log-XN}
For any positive integers $m,n$ and any real number $x>0$,
\begin{align}
\label{Eq-LP-Log-XN-2}
\frac{\log^m(x)}{m!(-1)^m}
&\,=1 +\sum\limits_{k=1}^m \sum\limits_{n=1}^\infty (-1)^n \frac{(2n+1)\ze_{n-1}(1_{k-1})\ze_{n+1}^\star(1_{m-k})}{n(n+1)} P_n(2x-1),\\
\frac{d^n}{dx^n} \left(\log^m (x)\right)
&\, =\frac{(-1)^n(n-1)!}{x^n} \sum\limits_{k=1}^m (-1)^k k!\binom{m}{k}\ze_{n-1}(1_{k-1}) \log^{m-k}(x). \label{Eq-LP-Log-XN}
\end{align}
\end{thm}
\begin{proof}
We first prove \eqref{Eq-LP-Log-XN}. Assume $n\ge 1$. Using \eqref{equ:Bell} we set
\[B_{n,k}:=B_{n,k}\left(x^{-1},-1!x^{-2},\ldots,(-1)^{n-k}(n-k)!x^{-(n-k+1)} \right).\]
By recurrence formula (\ref{3}) and by induction on $k$, we can prove easily that
\begin{align}
B_{n,k}=\frac{(-1)^{n-k}(n-1)!}{x^n} \ze_{n-1}(1_{k-1}).
\end{align}
We then apply Lemma \ref{lem1} with $f(t)=t^m$ and $g(x)=\log(x)$ to find that
\begin{align*}
f^{(k)}(g(x))= \left\{ {\begin{array}{*{20}{c}}k!\binom{m}{k}\log^{m-k}(x)
   {\ \ (k\leq m),}  \\
   \quad\quad\quad{0\quad\quad\quad\ \;\;\;(k>m).}  \\
\end{array} } \right.
\end{align*}
Further,
\begin{align}
\frac{d^n}{dx^n} \left(\log^m (x)\right)=\sum\limits_{k=0}^n f^{(k)}(g(x))B_{n,k}.
\end{align}
Hence, adding up these three contributions yields \eqref{Eq-LP-Log-XN}.

We now prove \eqref{Eq-LP-Log-XN-2}. By Fourier--Legendre series expansions, we have
\begin{align}\label{Eq-FL-Logx-not}
\log^m(x)=\sum_{n=0}^\infty \left\{(2n+1)\int_0^1 P_n(2x-1)\log^m(x)dx\right\}P_n(2x-1).
\end{align}
Using integration by parts and \eqref{Eq-LP-Log-XN}, we see that
\begin{align}\label{Eq-FL-Logxnot-2}
&\int_0^1 P_n(2x-1)\log^m(x)dx=\frac1{n!} \int_0^1 \left\{ \frac{d^n}{dx^n}\log^m(x)dx\right\} x^n(1-x)^ndx\nonumber\\
&=\frac{(-1)^n}{n} \sum\limits_{k=1}^m (-1)^k k!\binom{m}{k}\ze_{n-1}(1_{k-1})\int_0^1 (1-x)^n\log^{m-k}(x)dx \nonumber\\
&=\frac{(-1)^{m+n}m!}{n(n+1)} \sum\limits_{k=1}^{m} \ze_{n-1}(1_{k-1})\ze_{n+1}^\star(1_{m-k}),
\end{align}
where we have used the well-known identity (\cite[Eq. (2.5)]{Xu2017})
\begin{align}\label{Eq-x-n-1-Log1-x-IT}
\int_0^1 x^{n-1} \log^m(1-x)dx=(-1)^mm!\frac{\ze_n^\star(1_m)}{n}.
\end{align}
Thus, \eqref{Eq-LP-Log-XN-2} follows from \eqref{Eq-FL-Logx-not} and \eqref{Eq-FL-Logxnot-2} immediately.
We have now completed the proof of Thm.~\ref{thm-LP-Log-XN}.
\end{proof}

\section{Expansion of $\log^m(x)/\sqrt{x}$ and their higher derivatives}\label{sec-logm-sqrt}
In this section, we will apply Legendre polynomials to obtain an expansion of $\log^m(x)/\sqrt{x}$.
In order to do this, we will first find an expression of the higher derivatives of $\log^m(x)/\sqrt{x}$
using multiple $t$-harmonic sums in Thm.~\ref{thm-LP-Log}, which in turn requires us to use
some important facts of the Eulerian Beta function defined by
\begin{equation*}
B\left( {\alpha,\beta} \right) := \int_0^1 x^{\alpha - 1} (1-x)^{\beta - 1}\, dx
= \frac{\Gamma(\alpha)\Gamma(\beta)}{\Gamma(\alpha + \beta)}, \;
{\mathop{\Re}\nolimits}(\alpha)>0, {\mathop{\Re}\nolimits}(\beta)>0
\end{equation*}
and the digamma function defined by
\begin{equation*}
\psi(x) = \frac{\Gamma'(x)}{\Gamma(x)}
\end{equation*}
where $\Gamma(x)$ is the usual gamma function.

 From the definition, it is obvious that
\begin{equation*}
 \frac{\partial B(\alpha ,\beta)}{\partial \alpha}=B(\alpha ,\beta)[\psi(\alpha)-\psi(\alpha +\beta)].
\end{equation*}
Therefore, differentiating this equality $m-1$ times by the Leibniz rule, we can deduce that
\begin{align}\label{Eq-Beta-Relat}
\frac{\partial^m B(\alpha ,\beta)}{\partial \alpha^m} =
\sum_{i=0}^{m-1} \binom{m-1}{i} \frac{\partial^i B(\alpha ,\beta)}{\partial\alpha ^i}
\Big[\psi^{(m - i - 1)} (\alpha) - \psi^{(m - i - 1)} (\alpha+\beta) \Big].
\end{align}
Here, $\psi^{(m)}(x)$ stands for the polygamma function of order $m$ defined as the
$(m+1)$-st derivative of the logarithm of the gamma function:
\begin{equation*}
\psi^{(m)}(x):=\frac{d^m}{dx^m}\psi(x) = \frac{d^{m+1}}{dx^{m+1}}\log \Gamma(x).
\end{equation*}
Observe that $\psi^{(m)}(x)$ satisfy the following relations
\begin{equation*}
\psi(z)=-\gamma+ \sum_{n = 0}^\infty \left( \frac{1}{n + 1}-\frac{1}{n + z}\right),\ z\notin \N^-_0:=\{0,-1,-2,\ldots\},
\end{equation*}
\begin{equation*}
\psi^{(n)}(z) = (-1)^{n + 1} n!\sum_{k=0}^\infty  \frac1{(z+k)^{n+1}},\ n\in \N,
\end{equation*}
\begin{equation*}
\psi(x + n) = \frac{1}{x} + \frac{1}{x + 1} +  \cdots  + \frac{1}{x + n - 1} + \psi(x),\ n \in \N .
\end{equation*}
Here, $\gamma$ denotes the Euler-Mascheroni constant, defined by
\begin{equation*}
\gamma:=\lim _{n\to\infty} \left(\sum_{k=1}^n \frac{1}{k}-\log n\right)=-\psi(1) \approx  0. 57721566490153286.
\end{equation*}
Note that if $\alpha=1/2$, then
\begin{align}\label{Eq-Psi}
\psi^{(m-i-1)}(1/2+n)-\psi^{(m-i-1)}(1/2)=(-1)^{m-i-1}(m-i-1)!2^{m-i} t_n(m-i).
\end{align}
Hence, setting $\alpha=1/2$ and $\beta=n+1$ in \eqref{Eq-Beta-Relat} yields
\begin{align}\label{Eq-Beta-Relat-2}
\frac{(-1)^m}{2^mm!} \left. \frac{\partial^{m}B(\alpha ,\beta)}{\partial\alpha^m} \right|_{\alpha=1/2,\beta=n+1}=\frac{1}{m}\sum_{i=0}^{m-1}\frac{(-1)^i}{2^ii!}
\left.\frac{\partial^{i}B(\alpha,\beta)}{\partial\alpha ^i} \right|_{\alpha=1/2,\beta=n+1}t_{n+1}(m-i).
\end{align}
In particular, if $m=0$ then
\begin{equation}\label{BetaHalf}
 B\Big(\frac12,n+1\Big)=\frac{2\cdot 4^n}{(2n+1)\binn}.
\end{equation}

\begin{pro} For any positive integers $n$ and $k$,
\begin{align}\label{Eq-Sum-MHS}
a_n(k):=\frac{\ze_{n-1}(1_{k-1})}{n}+\sum_{i+j=n,\atop i,j\geq 1}\frac{\ze_{i-1}(1_{k-1})}{i} \frac{\binom{2j}{j}}{4^j}=\frac{1}{4^n}\binn t_n(1_k)2^k.
\end{align}
\end{pro}
\begin{proof} By the well-known binomial expansion,
\begin{align}\label{Eq-GF-sqrt1-x}
\frac{1}{\sqrt{1-x}}-1=\sum_{n=1}^\infty \frac{1}{4^n}\binn x^n,\quad x\in [-1,1).
\end{align}
Further, using the shuffle product we see that
\begin{align}
\log^k(1-x)=&(-1)^k\Li_1(x)^k=(-1)^kk!\int_0^x \left(\frac{dt}{1-t}\right)^k=(-1)^kk!\Li_{1_k}(x,1_{k-1})\notag\\
=&(-1)^kk!\sum_{n=1}^\infty \frac{\ze_{n-1}(1_{k-1})}{n}x^n,\quad x\in [-1,1).\label{Eq-GF-Log-k}
\end{align}
Hence, by the \emph{Cauchy product}, we obtain
\begin{align}\label{Eq-GF-Sums}
\frac{(-1)^k}{k!}\frac{\log^k(1-x)}{\sqrt{1-x}}&=\sum_{n=1}^\infty \left\{\frac{\ze_{n-1}(1_{k-1})}{n}+\sum_{i+j=n,\atop i,j\geq 1}\frac{\ze_{i-1}(1_{k-1})}{i} \frac{\binom{2j}{j}}{4^j}\right\}x^n\nonumber\\
&=\frac1{k!}\lim_{\alpha\to 1/2}\frac{\partial^k}{\partial\alpha^k}\frac1{(1-x)^\alpha}=\sum_{n=1}^\infty a_n(k)x^n.
\end{align}
Denote by $(\alpha)_n$ the Pochhammer symbol (or the shifted factorial) given by
\begin{equation*}
 (\alpha)_0:=1, \quad (\alpha)_n:=\alpha(\alpha+1)\cdots(\alpha+n-1),\ \forall n\ge 1.
\end{equation*}
Then, combining \eqref{Eq-GF-Sums} with the binomial series expansion $(1-x)^{-\alpha}=\sum_{n=0}^\infty (\alpha)_nx^n/n!, x\in(-1,1)$, we get
\begin{align}\label{Eq-CF}
a_n(k)=\frac{1}{k!n!}\lim_{\alpha\to 1/2} \frac{\partial^k}{\partial\alpha^k} (\alpha)_n.
\end{align}
In particular, when $k=0$ we obtain
\begin{align}\label{Eq-CF-0}
a_n(0)=\frac{1}{4^n}\binn.
\end{align}

By a simple calculation, the $\frac{\partial^k (\alpha)_n}{\partial x^k}$ satisfy a recurrence relation in the form of
\begin{align}\label{Eq-CF-1}
\frac{\partial^k (\alpha)_n}{\partial x^k}
= \sum_{i=0}^{k-1} \binom{k-1}{i}\frac{\partial^i (\alpha)_n}{\partial x^i}
\Big[ \psi^{(k-i-1)}(\alpha+n)-\psi^{(k-i-1)}(\alpha) \Big],\ k\in \N.
\end{align}
Hence, taking $m=k$ in \eqref{Eq-Psi} to evaluate \eqref{Eq-CF-1} and using \eqref{Eq-CF}, we find the recurrence relation
\begin{align}\label{Eq-CF-Relat}
a_n(k)=\frac{(-1)^{k-1}}{k}\sum_{i=0}^{k-1}(-1)^i2^{k-i} a_n(i)t_n(k-i).
\end{align}
In Lemma \ref{lem-S-2}, letting $x_k=\frac{1}{2k-1}$ yields
\begin{align}\label{Eq-CF-Relat-mtv}
t_n(1_m)=\frac{(-1)^{m-1}}{m}\sum_{i=0}^{m-1}(-1)^i t_n(1_i)t_n(m-i).
\end{align}
Multiplying \eqref{Eq-CF-Relat-mtv} by $\frac{2^m}{4^n}\binn$, then comparing it with \eqref{Eq-CF-Relat}, we obtain the result
\begin{align}\label{Eq-CF-Finally}
a_n(m)=\frac{2^m}{4^n}\binn t_n(1_m).
\end{align}
Finally, setting $m=k$ in \eqref{Eq-CF-Finally} yields the desired formula \eqref{Eq-Sum-MHS}.
\end{proof}

We an now consider the expansions of the higher derivatives of $\log^m(x)/\sqrt{x}$.
\begin{thm}\label{thm-LP-Log}
For any $m,n\in\N_0$ and any real number $x>0$,
\begin{align}\label{Eq-LP-Log}
\frac{d^n}{dx^n} \left(\frac{\log^m(x)}{\sqrt{x}} \right)=(-1)^n \frac{(2n)!}{n!4^n}\sum_{k=0}^m (-1)^k k!\binom{m}{k}2^kt_n(1_k)\frac{\log^{m-k}(x)}{x^{n+1/2}}.
\end{align}
\end{thm}
\begin{proof}
By a direct calculation, we have
\begin{align}\label{Eq-Diff-x}
\frac{d^n}{dx^n}\left(\frac1{\sqrt{x}}\right)=(-1)^n\frac{(2n)!}{n!4^n}x^{-1/2-n}.
\end{align}
Then, using \eqref{Eq-LP-Log-XN} and \eqref{Eq-Diff-x}, we can deduce by the Leibniz rule that
\begin{align}\label{Eq-Diff-x2}
&\frac{d^n}{dx^n} \left(\frac{\log^m (x)}{\sqrt{x}} \right)=\sum_{i+j=n,\atop i,j\geq 0} \frac{n!}{i!j!} \frac{d^i}{dx^i} \left(\log^m (x)\right) \frac{d^j}{dx^j}\left(\frac1{\sqrt{x}}\right)\nonumber\\
&=(-1)^n\frac{(2n)!}{n!4^n} \frac{\log^m (x)}{x^{n+1/2}}\nonumber\\&\quad+(-1)^nn!\sum_{k=1}^m (-1)^kk!\binom{m}{k} \left\{\frac{\ze_{n-1}(1_{k-1})}{n}+\sum_{i+j=n,\atop i,j\geq 1}\frac{\ze_{i-1}(1_{k-1})}{i} \frac{\binom{2j}{j}}{4^j} \right\}\frac{\log^{m-k}(x)}{x^{n+1/2}}.
\end{align}
Thus, combining \eqref{Eq-Diff-x2} with \eqref{Eq-Sum-MHS}, we can complete the proof of the theorem immediately.
\end{proof}

\begin{thm}\label{thm-LP-Log-2}
For any $m\in\N_0$ and any real number $x>0$,
\begin{align}\label{Eq-LP-Log-2}
\frac{\log^m (x)}{\sqrt{x}}=(-1)^mm!2^{m+1}\sum_{k=0}^m \sum_{n=0}^\infty (-1)^nt_n(1_k)t_{n+1}^\star(1_{m-k})P_n(2x-1).
\end{align}
\end{thm}
\begin{proof}
By Fourier--Legendre series expansions, we have
\begin{align}\label{Eq-FL-LogSqrtx}
\frac{\log^m(x)}{\sqrt{x}}=\sum_{n=0}^\infty \left\{(2n+1)\int_0^1 P_n(2x-1)\frac{\log^m(x)}{\sqrt{x}}dx\right\}P_n(2x-1).
\end{align}
Using integration by parts and \eqref{Eq-LP-Log}, we deduce that
\begin{align}\label{Eq-FL-LogSqrtx-2}
&\int_0^1 P_n(2x-1)\frac{\log^m(x)}{\sqrt{x}}dx=\frac1{n!} \int_0^1 \left\{ \frac{d^n}{dx^n}\frac{\log^m(x)}{\sqrt{x}}\right\} x^n(1-x)^ndx\nonumber\\
&=(-1)^n \frac{1}{4^n}\binn \sum_{k=0}^m (-1)^k k! \binom{m}{k} 2^k t_n(1_k)\int_0^1 x^{-1/2}(1-x)^n\log^{m-k}(x)dx.
\end{align}
Observe that
\begin{align}\label{Eq-Beta-Integral}
\int_0^1 x^{-1/2}(1-x)^n\log^{m}(x)dx={\left. {\frac{{{\partial ^{m}}B\left( {\alpha ,\beta } \right)}}{{\partial {\alpha ^m}}}} \right|_{\alpha=1/2,\beta  = n+1}}.
\end{align}
On the other hand, from Lemma \ref{lem-S-1}, changing $n$ to $n+1$ and letting $x_k=\frac1{2k-1}$, we obtain
\begin{align}\label{Eq-mtsv-Relat}
t_{n+1}^\star(1_m)=\frac{1}{m}\sum_{i=0}^{m-1}t_{n+1}^\star (1_i) t_{n+1}(m-i).
\end{align}
Hence, comparing \eqref{Eq-Beta-Relat-2} with \eqref{Eq-mtsv-Relat} gives
\begin{align}\label{Eq-mtsv-beta-Finall}
{\left. {\frac{{{\partial ^{m}}B\left( {\alpha ,\beta } \right)}}{{\partial {\alpha ^m}}}} \right|_{\alpha=1/2,\beta  = n+1}}=(-1)^mm!2^{m+1} \frac{4^n}{(2n+1)\binn} t_{n+1}^\star(1_m).
\end{align}
Thus, combining \eqref{Eq-FL-LogSqrtx-2}, \eqref{Eq-Beta-Integral} and \eqref{Eq-mtsv-beta-Finall}, we get
\begin{align}\label{Eq-FL-Finall}
\int_0^1 P_n(2x-1)\frac{\log^m(x)}{\sqrt{x}}dx=\frac{(-1)^{n+m}m!2^{m+1}}{2n+1}\sum_{k=0}^m  t_n(1_k)t_{n+1}^\star(1_{m-k}).
\end{align}
Finally, combining \eqref{Eq-FL-Finall} with \eqref{Eq-FL-LogSqrtx} we arrive at \eqref{Eq-LP-Log-2}.
This finishes the proof of the theorem.
\end{proof}

\begin{re} In fact, we can find the following more general results
\begin{align}\label{Eq-LP-Log-general}
\frac{d^n}{dx^n} \left(\frac{\log^m(x)}{x^\alpha} \right)=(-1)^n (\alpha)_n\sum_{k=0}^m (-1)^k k!\binom{m}{k}\ze_n(1_k;\alpha)\frac{\log^{m-k}(x)}{x^{n+\alpha}}
\end{align}
for $\alpha\neq 0,-1,-2,-3,\ldots$, and
\begin{align}\label{Eq-LP-Log-2-general}
\frac{\log^m (x)}{x^\alpha}=(-1)^mm!\sum_{k=0}^m \sum_{n=0}^\infty (-1)^n\frac{(2n+1)(\alpha)_n}{(1-\alpha)_{n+1}}\ze_n(1_k;\alpha)\ze_{n+1}^\star(1_{m-k};1-\alpha)P_n(2x-1)
\end{align}
for $\alpha\notin\Z$,
by a similar argument as in the proofs of \eqref{Eq-LP-Log-2} and \eqref{Eq-LP-Log}. Here, for a composition $\bfk=(k_1,\ldots,k_r)$ and $\alpha\neq 0,-1,-2,-3,\ldots$,
\begin{align*}
&\ze_n(\bfk;\alpha):=\sum\limits_{n\geq n_1>\cdots>n_r>0 } \frac{1}{(n_1+\alpha-1)^{k_1}\cdots (n_r+\alpha-1)^{k_r}},\\
&\ze^\star_n(\bfk;\alpha):=\sum\limits_{n\geq n_1\geq\cdots\geq n_r>0} \frac{1}{(n_1+\alpha-1)^{k_1}\cdots (n_r+\alpha-1)^{k_r}}.
\end{align*}
If $n<k$ then ${\ze_n}(\bfk;\alpha):=0$ and ${\ze_n}(\emptyset;\alpha)={\ze^\star _n}(\emptyset;\alpha):=1$.
\end{re}

\section{Ap\'{e}ry-type series involving central binomial coefficients}\label{sec-Apery}
Recall that the complete elliptic integral of the first kind is defined by
\begin{equation}\label{KEQ1}
K(x)=\int_0^{\pi/2} \frac{dt}{\sqrt{1-x \sin^2 t}}=\frac{\pi}{2} {}_2F_1\Big(\frac12, \frac12; 1 \Big| x\Big).
\end{equation}
By \cite[(8)]{Camp-Aur-Son19} we see that
\begin{align}\label{KEQ}
K(x)=\frac{\pi}{2}\sum\limits_{n=0}^\infty \bigg[\frac1{4^n}\binn\bigg]^2x^n=2\sum\limits_{n=0}^\infty \frac{P_n(2x-1)}{2n+1}.
\end{align}

In this section, using the above Fourier--Legendre expansion of $K(x)$ we will derive the
evaluation of some Ap\'{e}ry-type series involving squares of central binomial coefficients
in terms of colored MZVs.

\begin{thm}\label{thm-ASASs1}
For any $m\in\N_0$ we have
\begin{align}\label{ASASs1}
\sum\limits_{n=0}^\infty \bigg[\frac1{4^n}\binn\bigg]^2\frac1{(n+1)^{m+1}}
&\, =\frac{4}{\pi}\left\{1+\sum\limits_{k=1}^m \sum\limits_{n=1}^\infty \frac{\ze_{n-1}(1_{k-1})\ze_{n+1}^\star(1_{m-k})}{n(n+1)(2n+1)} (-1)^n \right\},\\
\label{ASASs2}
\sum_{n=0}^\infty \bigg[\frac1{4^n}\binn\bigg]^2\frac{\ze_{n+1}^\star(1_m)}{n+1}
&\, =\frac{4}{\pi}\left\{1+\sum\limits_{k=1}^m \sum\limits_{n=1}^\infty \frac{\ze_{n-1}(1_{k-1})\ze_{n+1}^\star(1_{m-k})}{n(n+1)(2n+1)} \right\}.
\end{align}
\end{thm}
\begin{proof}
Multiplying \eqref{KEQ} by $\log^m(x)$ and integrating over the interval $(0,1)$, we can prove \eqref{ASASs1} by a straight-forward calculation with the aid of  \eqref{Eq-FL-Logxnot-2}.

Similarly, by \cite[Eq. (2.5)]{Xu2017} we have
\begin{align}\label{LPS-Log-K1}
\int_0^1 K(x)\log^m(1-x)dx=(-1)^mm!\frac{\pi}{2}\sum_{n=0}^\infty \bigg[\frac1{4^n}\binn\bigg]^2\frac{\ze_{n+1}^\star(1_m)}{(n+1)}.
\end{align}
Now, multiplying \eqref{KEQ} by $\log^m(1-x)$ and integrating over the interval $(0,1)$, we see that
\begin{align}\label{LPS-Log-K}
\int_0^1 K(x)\log^m(1-x)dx=2\sum_{n=0}^\infty \frac 1{2n+1}\int_0^1 P_n(2x-1)\log^m(1-x)dx.
\end{align}
Since the Legendre polynomial satisfies $P_n(-x)=(-1)^nP(x)$ by \eqref{equ:LegendrePoly},
applying \eqref{Eq-FL-Logxnot-2} we obtain
\begin{align}\label{LPS-Log-2}
\int\limits_{0}^1 P_n(2x-1)\log^m(1-x)dx=\frac{(-1)^{m}m!}{n(n+1)} \sum\limits_{k=1}^{m} \ze_{n-1}(1_{k-1})\ze_{n+1}^\star(1_{m-k}).
\end{align}
Thus, \eqref{ASASs2} follows from \eqref{LPS-Log-K1} to \eqref{LPS-Log-2} immediately. This concludes the proof of the theorem.
\end{proof}

In fact, we have the following results.
\begin{pro}\label{pro-inCMZV24}
For any integers $m\ge k\ge 0$ we have
\begin{align*}
\sum\limits_{n=1}^\infty \frac{\ze_{n-1}(1_{k-1})\ze_{n+1}^\star(1_{m-k})}{n(n+1)(2n+1)}\in\CMZV^{\le2}\quad \text{and}\quad
\sum\limits_{n=1}^\infty \frac{\ze_{n-1}(1_{k-1})\ze_{n+1}^\star(1_{m-k})}{n(n+1)(2n+1)} (-1)^n\in \CMZV^{\le4},
\end{align*}
where $\CMZV^{\le N}$ is the $\Q$-span of CMZVs of level $\le N$. In particular, $\CMZV^{\le 1}=\Q+\MZV$
where $\MZV$ is the $\Q$-span of all MZVs.
\end{pro}

\begin{proof}
According to the definition of multiple harmonic (star) sums, we know that
\begin{equation*}
 \ze_{n-1}(1_{k-1})\ze_{n+1}^\star(1_{m-k})\in \Q\left[\frac1{n},\frac1{n+1},\ze_n(\cdots)\right].
\end{equation*}
Hence, to prove the proposition we need to evaluate the following sums
\begin{align}\label{Eq-2}
\sum\limits_{n=1}^\infty \frac{\ze_{n}(k_1,\ldots,k_r)}{n^p(n+1)^q(2n+1)}\quad \text{and}\quad
\sum\limits_{n=1}^\infty \frac{\ze_{n}(k_1,\ldots,k_r)}{n^p(n+1)^q(2n+1)} (-1)^n.
\end{align}
By induction of $p+q$ we get the following partial fraction expansion,
\begin{align}\label{Eq-3}
\frac{1}{n^p(n+1)^q(2n+1)}=\sum_{i=2}^p \frac{a_i}{n^i}+\sum_{j=2}^q \frac{b_j}{(n+1)^j}+\frac{c}{n(2n+1)}+\frac{d}{(n+1)(2n+1)},
\end{align}
where $a_i,b_j,c$ and $d$ are rational numbers. Obviously,
\begin{align}\label{Eq-4}
\sum\limits_{n=1}^\infty \frac{\ze_{n}(k_1,\ldots,k_r)}{n^i}\in \MZV\quad\text{and}\quad \sum\limits_{n=1}^\infty \frac{\ze_{n}(k_1,\ldots,k_r)}{(n+1)^j}\in \MZV.
\end{align}
Moreover, noting the fact that
\begin{align}
&\ze_n(s_1,s_2,\ldots,s_m)\nonumber\\
&\quad=2^{s_1+s_2+\cdots+s_m-m}\sum_{2n\geq n_1>n_2>\cdots>n_m\geq1}
    \frac{(1+(-1)^{n_1})(1+(-1)^{n_2})\cdots(1+(-1)^{n_m})}
        {n_1^{s_1}n_2^{s_2}\cdots n_m^{s_m}}\nonumber\\
&\quad=2^{s_1+s_2+\cdots+s_m-m}\sum_{\substack{\si_j\in\{\pm 1\}\\j=1,2,\ldots,m}}
      \ze_{2n}(s_1,\ldots,s_m;\si_1,\dots,\si_m)\,.\label{zen.ze2n}
\end{align}
Hence, we have
\begin{align}
\sum\limits_{n=1}^\infty \frac{\ze_{n}(k_1,\ldots,k_r)}{n(2n+1)}&=2^{k_1+k_2+\cdots+k_r-r+1}\sum_{\substack{\si_j\in\{\pm 1\}\\j=1,2,\ldots,r}}\sum\limits_{n=1}^\infty \frac{\ze_{2n}(s_1,\ldots,s_r;\si_1,\dots,\si_r)}{2n(2n+1)}\nonumber\\
&=2^{k_1+k_2+\cdots+k_r-r}\sum_{\substack{\si_j\in\{\pm 1\}\\ j=1,2,\ldots,r}}
    \sum_{n=1}^\infty \frac{\ze_{n}(s_1,\ldots,s_m;\si_1,\dots,\si_m)}{n(n+1)}(1+(-1)^n).
\end{align}
Then, applying \cite[Thms.~3.1]{XuZhao2020b} (by taking $n=1$ there), we obtain
\begin{align}
\sum\limits_{n=1}^\infty \frac{\ze_{n}(k_1,\ldots,k_r)}{n(2n+1)}\in \CMZV^{\le2}.
\end{align}
Similarly, we have
\begin{align}
\sum\limits_{n=1}^\infty \frac{\ze_{n}(k_1,\ldots,k_r)}{(n+1)(2n+1)}\in \CMZV^{\le2}.
\end{align}
Hence,
\begin{align}
\sum\limits_{n=1}^\infty \frac{\ze_{n-1}(1_{k-1})\ze_{n+1}^\star(1_{m-k})}{n(n+1)(2n+1)}\in\CMZV^{\le2}.
\end{align}
We can prove the second formula in the proposition in a similar fashion.
\end{proof}

\begin{cor}
For any $m\in\N_0$ we have
\begin{align*}
\sum\limits_{n=0}^\infty \bigg[\frac1{4^n}\binn\bigg]^2\frac1{(n+1)^{m+1}}\in \frac1{\pi} \CMZV^{\le4}, \qquad
\sum_{n=0}^\infty \bigg[\frac1{4^n}\binn\bigg]^2\frac{\ze_{n+1}^\star(1_m)}{n+1}\in \frac1{\pi} \CMZV^{\le2}.
\end{align*}
\end{cor}

\begin{exa}\label{eg:mixedWt}
By straight-forward calculations and Au's package \cite{Au2020},
we have the following cases (which can also be found in \cite{Camp-Aur-Son19,Can-Aur2019})
\begin{align*}
&\sum_{n=0}^\infty \bigg[\frac1{4^n}\binn\bigg]^2\frac1{n+1}=\frac4{\pi},\\
&\sum_{n=0}^\infty \bigg[\frac1{4^n}\binn\bigg]^2\frac1{(n+1)^2}=\frac{16}{\pi}-4,\\
&\sum_{n=0}^\infty \bigg[\frac1{4^n}\binn\bigg]^2\frac1{(n+1)^3}=\frac{16}{\pi}(3-2G-\pi+\pi\log 2),\\
&\sum_{n=0}^\infty \bigg[\frac1{4^n}\binn\bigg]^2\frac{\ze_{n+1}^\star(1)}{n+1}=\frac4{\pi}(1-\log 2),\\
&\sum_{n=0}^\infty \bigg[\frac1{4^n}\binn\bigg]^2\frac{\ze_{n+1}^\star(1,1)}{n+1}=\frac4{\pi}\big(12-2\ze(2)-16\log 2+8\log^2(2)\big),
\end{align*}
and
\begin{equation*}
\sum_{n=0}^\infty \bigg[\frac1{4^n}\binn\bigg]^2\frac1{(n+1)^4}=
\frac{128}{\pi}\left(1-G-2{\rm Im}\Li_3\Big(\frac{1+i}{2}\Big)  \right)
-48+6\pi^2+64\log 2-24 \log^2(2),
\end{equation*}
where $i=\sqrt{-1}$ and $G:=\sum_{n=0}^\infty \frac{(-1)^n}{(2n+1)^2}$ is \emph{Catalan's constant}. Note that
$\Li_3\Big(\frac{1+i}{2}\Big)\in\CMZV^4$ by the next theorem which is of independent interest.
\end{exa}

\begin{thm}
For all $\bfk\in\N^r$ we have
\begin{equation*}
\Li_\bfk\Big(\frac{1+i}{2}\Big),\ \Li_\bfk\Big(\frac{1-i}{2}\Big)\in\CMZV_{|\bfk|}^4.
\end{equation*}
\end{thm}
\begin{proof} We know that
$$
\Li_\bfk\Big(\frac{1+i}{2}\Big)=\int_0^{\text{$\scriptstyle \frac{1+i}{2}$}} \frac{dt}{1-t}\left(\frac{dt}{t}\right)^{k_r-1}\cdots \frac{dt}{1-t}\left(\frac{dt}{t}\right)^{k_1-1}
$$
by using Chen's iterated integral expression (see \cite[pp.~13-14]{Zhao2016} where the convention of the 1-form order
is opposite to the one used in this paper).
By the change of variables $t\to \frac{1+i}{2}(1-t)$ we see that $d\log t\to d\log(1-t)=\frac{-dt}{1-t}$ and
$$\frac{-dt}{1-t}=d\log(1-t)\to d\log \left(\frac{1-i}{2}+\frac{1+i}{2}t\right)=d\log (-i+t)=\frac{-dt}{i-t}.$$
Thus
\begin{equation*}
\Li_\bfk\Big(\frac{1+i}{2}\Big)=(-1)^r \int_0^{1}\left(\frac{dt}{1-t}\right)^{k_1-1} \frac{dt}{i-t}\cdots \left(\frac{dt}{1-t}\right)^{k_r-1}\frac{dt}{i-t}\in\CMZV_{|\bfk|}^4.
\end{equation*}
Taking complex conjugation we see that $\Li_\bfk\Big(\frac{1-i}{2}\Big)\in\CMZV_{|\bfk|}^4.$
\end{proof}

\section{Some special MtV variant of Ap\'{e}ry-type series}\label{sec-Apery-MtV}
In this section, by applying Thm.~\ref{thm-LP-Log-2} and using the Fourier--Legendre expansion of $K(x)$ we will derive explicit
evaluations of some MtV variant of Ap\'{e}ry-type series involving squares of central binomial coefficients.

In order to prove the next theorem, we will need the following result.

\begin{pro} \label{pro-MtV}
Let $\eta=\pm 1$. For all $m,l\in\Z$ with $l\ge 2$ and $m\ge 0$, we have
\begin{align*}
     \sum_{k=0}^m \sum_{n\ge 0} \eta^n \frac{\ze_n(1_{m-k})\ze_{n+1}^\star(1_k)}{(n+1)^l}
=&\, \sum_{d=1}^{m+1}\sum_{|\bfs|=m+2,s_1\ge2} 2^{d-1} \ze(s_1+l-2,s_2,\dots,s_d;\eta,1_{d-1}),\\
     \sum_{k=0}^m \sum_{n\ge 0} \eta^n \frac{t_n(1_{m-k}) t_{n+1}^\star(1_k)}{(2n+1)^l}
=&\, \sum_{d=1}^{m+1}\sum_{|\bfs|=m+2,s_1\ge2} 2^{d-1} t(s_1+l-2,s_2,\dots,s_d;\eta,1_{d-1}),
\end{align*}
where $\bfs=(s_1,\dots,s_d).$
\end{pro}
\begin{proof} Fix any positive integer $n$. Then both $\ze_n$ and $t_n$ satisfy stuffle relations.
We now prove the claims by induction on $m$. If $m=1$ then both sides are
equal to $f(l+1;\eta)+2f(l,1;\eta,1)$ for $f=\ze$ or $f=t$. In general, we see that
\begin{align}\notag
&   \sum_{k=0}^m \sum_{n\ge 0}\eta^n \frac{\ze_n(1_{m-k})\ze_{n+1}^\star(1_k)}{(n+1)^l}\\
=&\,\sum_{k=1}^m \sum_{n\ge 0}\eta^n\frac{\ze_n(1_{m-k})\ze_{n+1}^\star(1_{k-1})}{(n+1)^{l+1}}
+\sum_{k=0}^m \sum_{n\ge 0}\eta^n \frac{\ze_n(1_{m-k})\ze_n^\star(1_k)}{(n+1)^l}\notag \\
=&\,  \sum_{d=1}^{m}\sum_{|\bfs|=m+1,s_1\ge2} 2^{d-1} \ze(s_1+l-1,s_2,\dots,s_d;\eta,1_{d-1})
+\sum_{k=0}^m \sum_{n\ge 0}\eta^n \frac{\ze_n(1_{m-k})\ze_n^\star(1_k)}{(n+1)^l} \label{equ:zetaVersion}
\end{align}
by induction. By the following lemma about words and noticing that
\begin{equation*}
  f_n^\star(1_k)=\sum_{j=1}^k \sum_{|\bfr|=k,\bfr\in\N^j} f_n(\bfr)
\end{equation*}
for $f=\ze$ or $f=t$, we see that
\begin{align*}
\eqref{equ:zetaVersion}=&\, \sum_{d=1}^{m}\sum_{\substack{|\bfs|=m+1\\  s_1\ge2,\bfs\in\N^d}} 2^{d-1} \ze(s_1+l-1,s_2,\dots,s_d;\eta,1_{d-1})
+\sum_{d=0}^m \sum_{\substack{|\bfs|=m\\ \bfs\in\N^d}} 2^d \ze(l,s_1,s_2,\dots,s_d;\eta,1_{d-1})\\
=&\,  \sum_{d=1}^{m+1}\sum_{\substack{|\bfs|=m+2\\  s_1\ge 3,\bfs\in\N^d}} 2^{d-1} \ze(s_1+l-2,s_2,\dots,s_d;\eta,1_{d-1}) \\
+&\, \sum_{d=1}^{m+1} \sum_{\substack{|\bfs|=m+2\\ s_1=2,\bfs\in\N^d}} 2^{d-1} \ze(s_1+l-2,s_2,\dots,s_d;\eta,1_{d-1}).
\end{align*}
Similar identity holds for $t$, too. The proposition follows quickly.
\end{proof}

\begin{lem} Let $({\frak A}^1,*)$ be Hoffman's $\Q$-algebra of words generated by $z_j (j\in\N)$ satisfying
the stuffle relations (see \cite{Hoffman2000}). For any $\bfs=(s_1,\dots,s_d)\in\N^d$ we put $z_\bfs=z_{s_1}z_{s_2}\dots z_{s_d}$. Then for all non-positive integer $m$ we have
\begin{equation*}
\left(\sum_{k=0}^m \sum_{q=0}^{k} \sum_{|\bfr|=k,\bfr\in\N^q} z_1^{m-k}*z_\bfr=\right)
\sum_{k=0}^m  \sum_{|\bfr|=k} z_1^{m-k}*z_\bfr= \sum_{|\bfs|=m} 2^{\dep(\bfs)} z_\bfs
\left(=\sum_{d=0}^m \sum_{|\bfs|=m,\bfs\in\N^d} 2^d z_\bfs\right).
\end{equation*}
\end{lem}
\begin{proof}
We proceed by induction on $m$.
When $m=0$ both sides of the equation in the lemma are empty word. Assume $m\ge 1$. For
any $\bfr\in\N^q$, by the recursive definition of the stuffle product, we have (after
setting $a=r_1$ and $\bfr=(r_2,\dots,r_q)$)
\begin{align*}
\sum_{k=0}^m \sum_{|\bfr|=k} z_1^{m-k}*z_\bfr
=&\, \sum_{k=0}^{m-1}\sum_{|\bfr|=k}  z_1(z_1^{m-1-k}*z_\bfr)
 +\sum_{a=1}^m \sum_{k=a}^{m} \sum_{|\bfr|=k-a} z_{a}(z_1^{m-k}*z_{\bfr})\\
 &\, +\sum_{a=1}^{m-1} \sum_{k=a}^{m-1} \sum_{|\bfr|=k-a} z_{1+a}(z_1^{m-1-k}*z_{\bfr}).
\end{align*}
Breaking the middle term into two cases (namely, $a=r_1=1$ and $a=r_1\ge 2$) and shifting
the indices $k$ and $a$ suitably, we easily get
\begin{align*}
&\, \sum_{k=0}^m \sum_{|\bfr|=k} z_1^{m-k}*z_\bfr
=\, 2\sum_{k=0}^{m-1}\sum_{|\bfr|=k}  z_1(z_1^{m-1-k}*z_\bfr)
 +2\sum_{a=2}^m \sum_{k=0}^{m-a} \sum_{|\bfr|=k} z_{a}(z_1^{m-a-k}*z_{\bfr})\\
=&\,2\sum_{a=1}^m \sum_{d=0}^{m-a} \sum_{|\bfs|=m-a} 2^{\dep(\bfs)} z_az_\bfs
\end{align*}
by inductive assumption. The lemma follows immediately.
\end{proof}

The following theorem and \eqref{ASASs1} can be regarded as explicit versions of \cite[Thm. 3]{Can-Aur2019}.

\begin{thm}\label{thm-MR-Apery-AMtVs}
For any $m\in\N_0$ we have
\begin{align}\label{MR-Apery-AMtVs}
\sum\limits_{n=0}^\infty \bigg[\frac1{4^n}\binn\bigg]^2 \frac1{(2n+1)^{m+1}}
=&\, \frac{4}{\pi}\sum_{k=0}^m \sum_{n=0}^\infty (-1)^n\frac{t_n(1_k)t_{n+1}^\star(1_{m-k})}{(2n+1)^2}\\
=&\, -\frac{2}{\pi} \sum_{d=1}^{m+1}\sum_{|\bfs|=m+2,\bfs\in\N^d, s_1\ge 2} 2^d t(\bfs;-1,1_{d-1})\label{thm-MR-Apery-AMtVs2} \\
\in&\,  \frac1{\pi} \AMtV_{m+2}\subset \frac1{\pi}\CMZV_{m+2}^4, \notag
\end{align}
where $\AMtV_w$ (resp. $\CMZV_w^N$) denotes the
$\Q$-span of all the alternating multiple $t$-values (resp. CMVZs of level $N$ and) of weight $w$.
\end{thm}

\begin{proof}
We consider the integral
\begin{align}\label{Eq-K-Log}
\int_0^1 K(x)\frac{\log^m(x)}{\sqrt{x}}dx&=2^{m+1} \int_0^1 K(x^2)\log^m(x)dx=2^{m}\pi \sum_{n=0}^\infty \bigg[\frac1{4^n}\binn\bigg]^2 \int_0^1 x^{2n} \log^m(x)dx\nonumber\\
&=2^mm!(-1)^m\pi \sum_{n=0}^\infty \bigg[\frac1{4^n}\binn\bigg]^2 \frac1{(2n+1)^{m+1}}\nonumber\\
&=2\sum_{n=0}^\infty \frac1{2n+1}\int_0^1 P_n(2x-1)\frac{\log^m(x)}{\sqrt{x}}dx\nonumber\\
&=2^{m+2}m!(-1)^m\sum_{k=0}^m \sum_{n=0}^\infty (-1)^n\frac{t_n(1_k)t_{n+1}^\star(1_{m-k})}{(2n+1)^2},
\end{align}
where we have used \eqref{Eq-FL-Finall} and \eqref{KEQ}.
Hence, we can prove \eqref{MR-Apery-AMtVs} by a straight-forward calculation. Finally, \eqref{thm-MR-Apery-AMtVs2}
follows immediately from Prop.\ \ref{pro-MtV}.
\end{proof}

\begin{exa}
Observe that for any $\bfs\in \N^d$, we have the explicit expression in $\CMZV^4$
\begin{equation*}
2^d t(\bfs;-1,1_{d-1})=i\sum_{\eta_1=\pm 1}\cdots \sum_{\eta_d=\pm 1} \eta_1\cdots \eta_d \Li_\bfs(i\eta_1,\eta_2,\dots, \eta_d).
\end{equation*}
For example,
\begin{align}
\sum\limits_{n=0}^\infty \bigg[\frac1{4^n}\binn\bigg]^2 \frac1{2n+1}&=-\frac{4t(\bar 2)}{\pi}=\frac{4G}{\pi}\approx 1.166243616,\label{4GI}\\
\sum\limits_{n=0}^\infty \bigg[\frac1{4^n}\binn\bigg]^2 \frac1{(2n+1)^2}&=-\frac{4}{\pi}\Big(t(\bar3)+2t(\bar 2,1)\Big) \notag \\
&=\frac{3\pi^2}{8}+\frac{\log^2(2)}{2}-\frac{16}{\pi}{\rm Im}\Li_3\left(\frac{1+i}{2}\right)\approx  1.037947765, \label{4GIm=2}\\
\sum\limits_{n=0}^\infty \bigg[\frac1{4^n}\binn\bigg]^2 \frac1{(2n+1)^3}
&=-\frac{4}{\pi}\Big(t(\bar4)+2t(\bar3,1)+2t(\bar 2,2)+4t(\bar2,1,1)\Big) \notag\\
&=\frac34 \pi^2 \log 2+\frac13\log^3(2)+\frac{64}{\pi}{\rm Im}\Li_4\left(\frac{1+i}{2}\right)-\frac{48}{\pi}G\approx 1.010879510,  \notag
\end{align}
where $G$ is Catalan's constant as before and $G(4):=\sum_{n=0}^\infty \frac{(-1)^n}{(2n+1)^4}$.
\end{exa}

\begin{thm}\label{thm-Apery-MRV}
For any $m\in\N_0$ and $k\in\N$, we have
\begin{align}\label{Eq-MtHS-CB}
\sum\limits_{n=1}^\infty \frac{1}{4^n}\binn\frac{t_n(1_k)}{n^{m+1}}=\frac1{2^k}R(k+1,1_m)\in\Q[\log 2,\ze(2),\ze(3),\ze(4),\ldots].
\end{align}
\end{thm}
\begin{proof}
From \eqref{Eq-GF-Sums} and \eqref{Eq-CF-Finally}, we deduce
\begin{align}\label{Eq-GF-Log-Sqrt1-x}
\frac{1}{2^k\sqrt{1-x}}\int_0^x\left(\frac{dt}{1-t}\right)^k =\frac{(-1)^k}{2^kk!}\frac{\log^k(1-x)}{\sqrt{1-x}}=\sum_{n=1}^\infty \frac{1}{4^n}\binn t_n(1_k)x^n.
\end{align}
If $k=0$ we need to modify this as follows:
\begin{align}\label{Eq-GF-Log-Sqrt1-x0}
\frac{1}{\sqrt{1-x}}-1=\sum_{n=1}^\infty \frac{1}{4^n}\binn x^n.
\end{align}
Multiplying \eqref{Eq-GF-Log-Sqrt1-x} by $\frac{\log^m(x)}{x}$ and integrating over $(0,1)$ yields
\begin{align}\label{Eq-Inte-Log-Sqrt1-x}
\sum\limits_{n=1}^\infty \frac{1}{4^n}\binn\frac{t_n(1_k)}{n^{m+1}}=\frac{(-1)^{m+k}}{2^km!k!}\int_0^1 \frac{\log^m(x)\log^k(1-x)}{x\sqrt{1-x}}dx.
\end{align}
In \cite{Xu2021}, the first author proved that
\begin{align}\label{Eq-Inte-MRV-H1}
{R}(m+1,\underbrace{1,\ldots,1}_{n-1})=\frac{(-1)^{m+n-1}}{m!(n-1)!}\int_{0}^1 \frac{\log^m(t)\log^{n-1}(1-t)}{1-t}t^{-1/2}dt.
\end{align}
Hence, applying the change of variables $x\to 1-x$ in the integral on the right-hand side of \eqref{Eq-Inte-Log-Sqrt1-x} and using \eqref{Eq-Inte-MRV-H1}, we obtain the desired formula
\begin{align*}
\sum\limits_{n=1}^\infty \frac{1}{4^n}\binn\frac{t_n(1_k)}{n^{m+1}}=\frac1{2^k}R(k+1,1_m).
\end{align*}
We can now finish the proof of Thm.~\ref{thm-Apery-MRV} by applying \eqref{Eq-MRV-H1}.
\end{proof}

\begin{exa} Taking $m\le 2$, we get
\begin{align*}
&\sum\limits_{n=1}^\infty \frac{1}{4^n}\binn\frac{t_n(1_k)}{n}=\frac{2^{k+1}-1}{2^k}\ze(k+1),\\
&\sum\limits_{n=1}^\infty \frac{1}{4^n}\binn\frac{t_n(1)}{n^2}=\frac{7}{2}\ze(3)-3\ze(2)\log 2,\\
&\sum\limits_{n=1}^\infty \frac{1}{4^n}\binn\frac{t_n(1,1)}{n^2}=\frac{45}{16}\ze(4)-\frac7{2}\ze(3)\log 2,\\
&\sum\limits_{n=1}^\infty \frac{1}{4^n}\binn\frac{t_n(1)}{n^3}=\frac{15}{4}\ze(4)+3\log^2(2)\ze(2)-7\ze(3)\log 2.
\end{align*}
In fact, from \cite[Thm. 3.4]{WX2021} we know that the ``$\in$'' part (but not the ``='' part as $R(1,\dots)$ is undefined) of the Thm.~\ref{thm-Apery-MRV} holds for $k=0$, too.
\end{exa}

\section{Some special product variant of Ap\'{e}ry-type series}\label{sec-Apery-MtSV}
In this section, we use the idea similar to the one in section \ref{sec-Apery} to derive explicit
evaluations of some product variants of Ap\'{e}ry-type series involving the central binomial coefficients.
By product variant we mean products of multiple harmonic sums and multiple $t$-harmonic sums (or
their star version) appear in the terms of such series.

\begin{thm}\label{thm-Apery-MtV-Mhs-CB}
For any $m,k\in\N$ we have
\begin{align}\label{Eq-Apery-MtV-Mhs-CB}
\sum_{n=1}^\infty  \frac{4^n}{\binn} \frac{\ze_{n-1}(1_{m-1})t^\star_n(1_k)}{n^2}&=\sum_{n=1}^\infty \frac{1}{4^n}\binn\frac{t_n(1_k)\ze^\star_n(1_m)}{n}\nonumber\\
&=\binom{m+k}{k}\frac{2^{m+k+1}-1}{2^k}\ze(m+k+1),
\end{align}
where if $m=0$, the second equal sign also holds.
\end{thm}

\begin{proof}
Applying \eqref{Eq-mtsv-beta-Finall} and noting that $B(\alpha,\beta)=B(\beta,\alpha)$, by a simple calculation, we have
\begin{align}\label{Eq-n-k-Log-Sqrtx}
\int_0^1 x^{n-1}\frac{\log^k(1-x)}{\sqrt{1-x}}dx=(-1)^kk!2^k \frac{4^n}{n\binn}t_n^\star(1_k).
\end{align}
Applying \eqref{Eq-x-n-1-Log1-x-IT}, \eqref{Eq-GF-Log-k}, \eqref{Eq-GF-Log-Sqrt1-x}, \eqref{Eq-Inte-MRV-H1} and \eqref{Eq-n-k-Log-Sqrtx}, we may compute the integral
\begin{align}\label{Eq-m-k-x-Sqrtx}
\int_0^1 \frac{\log^{m+k}(1-x)}{x\sqrt{1-x}}dx&=(-1)^{k+m}(m+k)!R(m+k+1)\nonumber\\&=(-1)^kk!2^k\sum_{n=1}^\infty \frac{1}{4^n}\binn t_n(1_k)\int_0^1 x^{n-1}\log^m(1-x)dx\nonumber\\
&=(-1)^{k+m}k!m!2^k \sum_{n=1}^\infty \frac{t_n(1_k)\ze^\star_n(1_m)}{n4^n}\binn\nonumber\\
&=(-1)^mm!\sum_{n=1}^\infty \frac{\ze_{n-1}(1_{m-1})}{n}\int_0^1 x^{n-1}\frac{\log^k(1-x)}{\sqrt{1-x}}dx\nonumber\\
&=(-1)^{k+m}k!m!2^k \sum_{n=1}^\infty \frac{\ze_{n-1}(1_{m-1})t^\star_n(1_k)}{n^2\binn}4^n.
\end{align}
Finally, we may use $R(k)=(2^k-1)\ze(k)\ (k>1)$ to get the desired formula.
\end{proof}

\begin{exa} Taking $m\le 2$ and $k\le 2$, we get
\begin{align*}
&\sum_{n=1}^\infty  \frac{4^n}{\binn}  \frac{t^\star_n(1)}{n^2}=\sum_{n=1}^\infty \frac{t_n(1)\ze^\star_n(1)}{n4^n}\binn=7\ze(3),\\
&\sum_{n=1}^\infty  \frac{4^n}{\binn} \frac{t^\star_n(1,1)}{n^2}=\sum_{n=1}^\infty \frac{t_n(1,1)\ze^\star_n(1)}{n4^n}\binn=\frac{45}{4}\ze(4),\\
&\sum_{n=1}^\infty \frac{4^n}{\binn}  \frac{\ze_{n-1}(1)t^\star_n(1)}{n^2}=\sum_{n=1}^\infty \frac{t_n(1)\ze^\star_n(1,1)}{n4^n}\binn=\frac{45}{2}\ze(4).
\end{align*}
\end{exa}

Furthermore, from \eqref{Eq-Apery-MtV-Mhs-CB}, changing $(m,k)$ to $(k,m)$, we obtain the duality relation
\begin{align}\label{Eq-Apery-MtV-Mhs-CB-dual}
&\, 2^k\sum_{n=1}^\infty  \frac{4^n}{\binn} \frac{\ze_{n-1}(1_{m-1})t^\star_n(1_k)}{n^2}=2^m\sum_{n=1}^\infty \frac{4^n}{\binn} \frac{\ze_{n-1}(1_{k-1})t^\star_n(1_m)}{n^2}\nonumber\\
=&\, 2^k\sum_{n=1}^\infty  \frac1{4^n}\binn \frac{t_n(1_k)\ze^\star_n(1_m)}{n} =2^m\sum_{n=1}^\infty  \frac1{4^n}\binn \frac{t_n(1_m)\ze^\star_n(1_k)}{n} .
\end{align}

\begin{thm}\label{thm-Apery-Mtshs-CB} For any $m\in\N_0$ we have
\begin{align}\label{Eq-Apery-Mtshs-CB}
\sum_{n=0}^\infty \frac{1}{4^n}\binn\frac{t_{n+1}^\star(1_m)}{2n+1}
=&\,\frac4{\pi}\sum_{k=0}^m \sum_{n=0}^\infty \frac{t_n(1_k)t_{n+1}^\star(1_{m-k})}{(2n+1)^2} \\
=&\, \frac4{\pi} \sum_{d=1}^{m+1}\sum_{\substack{\bfs=(s_1,\dots,s_d)\in\N^d \\ |\bfs|=m+2,s_1\ge 2}} 2^{d-1} t(\bfs)\in \frac1{\pi} \MtV_{m+2}\subset \frac1{\pi}\CMZV_{m+2}^2. \label{Eq-Apery-Mtshs-CB2}
\end{align}
\end{thm}
\begin{proof}
Applying $x\to 1-x$ in \eqref{KEQ} and using $P_n(-x)=(-1)^nP_n(x)$ we get
\begin{align}\label{KEQ-1}
K(1-x)=\frac{\pi}{2}\sum\limits_{n=0}^\infty \bigg[\frac1{4^n}\binn\bigg]^2(1-x)^n=2\sum\limits_{n=0}^\infty (-1)^n \frac{P_n(2x-1)}{2n+1}.
\end{align}
We consider the integral
\begin{align}\label{KEQ-1-Log-Sqrtx}
\int_0^1 K(1-x) \frac{\log^m(x)}{\sqrt{x}}dx&= \frac{\pi}{2} \sum_{n=0}^\infty \bigg[\frac1{4^n}\binn\bigg]^2 \int_0^1 (1-x)^{n} \frac{\log^m(x)}{\sqrt{x}}dx\nonumber\\
&=\frac{\pi}{2} \sum_{n=0}^\infty \bigg[\frac1{4^n}\binn\bigg]^2 \int_0^1 x^{n} \frac{\log^m(1-x)}{\sqrt{1-x}}dx\nonumber\\
&=(-1)^m m! 2^m \pi\sum_{n=0}^\infty \frac{1}{4^n}\binn\frac{t_{n+1}^\star(1_m)}{2n+1}\nonumber\\
&=2\sum_{n=0}^\infty \frac{(-1)^n}{2n+1}\int_0^1 P_n(2x-1)\frac{\log^m(x)}{\sqrt{x}}dx\nonumber\\
&=2^{m+2}m!(-1)^m\sum_{k=0}^m \sum_{n=0}^\infty \frac{t_n(1_k)t_{n+1}^\star(1_{m-k})}{(2n+1)^2},
\end{align}
where we have used formulas \eqref{Eq-FL-Finall} and \eqref{Eq-n-k-Log-Sqrtx}.
Finally, the equation \eqref{Eq-Apery-Mtshs-CB2} in the theorem follows from the Prop.\ \ref{pro-MtV}
by taking $l=2$. This completes the proof of the theorem.
\end{proof}

\begin{exa} Taking $m\le 3$, we can verify easily that
\begin{align*}
&\sum_{n=0}^\infty \frac{1}{4^n}\binn\frac{1}{2n+1}=\frac4{\pi}t(2)=\frac3{\pi}\ze(2)=\frac{\pi}{2},\\
&\sum_{n=0}^\infty \frac{1}{4^n}\binn\frac{t_{n+1}^\star(1)}{2n+1}=\frac4{\pi}\big(2t(2,1)+t(3)\big)=\pi \log 2,\\
&\sum_{n=0}^\infty \frac{1}{4^n}\binn\frac{t_{n+1}^\star(1,1)}{2n+1}=\frac4{\pi}\big(4t(2,1,1)+2t(3,1)+2t(2,2)+t(4)\big)=\frac{\pi^3}{24}+\pi \log^2(2),\\
&\sum_{n=0}^\infty \frac{1}{4^n}\binn\frac{t_{n+1}^\star(1,1,1)}{2n+1}
=\frac{\pi}{4}\ze(3)+\frac{2\pi}{3}\log^3(2)+\frac{\pi}{12}\log^2(2).
\end{align*}
\end{exa}

\section{General MSV-MtV product variant of Ap\'{e}ry-type series, I}\label{sec-SpecialMZSV-MtVproductI}
The theory of iterated integrals was developed first by K.T. Chen in the 1960's \cite{KTChen1971,KTChen1977}. It has played important roles in the study of algebraic topology and algebraic geometry in the past half century. For real values $a,b$ its simplest form is
$$\int_{a}^b f_p(t)dtf_{p-1}(t)dt\cdots f_1(t)dt:=\int\limits_{a<t_p<\cdots<t_1<b}f_p(t_p)f_{p-1}(t_{p-1})\cdots f_1(t_1)dt_1dt_2\cdots dt_p.$$
In this section, we use the iterated integrals to establish a recurrence relation of Ape\'ry series
and then derive the evaluations of some MZSV-MtV product variants of Ap\'{e}ry-type series.

Recall that the Hoffman dual of a composition $\bfk=(k_1,\ldots,k_r)$ is $\bfk^\vee=(k'_1,\ldots,k'_{r'})$ determined by $|\bfk|:=k_1+\cdots+k_r=k'_1+\cdots+k'_{r'}$ and
\begin{equation*}
\{1,2,\ldots,|\bfk|-1\}
=\Big\{ \sideset{}{_{i=1}^{j}}\sum k_i\Big\}_{j=1}^{r-1}
 \coprod \Big\{ \sideset{}{_{i=1}^{j}}\sum  k_i'\Big\}_{j=1}^{r'-1}.
\end{equation*}
Equivalently, $\bfk^\vee$ can be obtained from $\bfk$ by swapping the commas ``,'' and the plus signs ``+'' in the expression
\begin{equation}\label{eq:HdualDef}
 \bfk=(\underbrace{1+\cdots+1}_{\text{$k_1$ times}},\dotsc,\underbrace{1+\cdots+1}_{\text{$k_r$ times}}).
\end{equation}
For example, we have
$({1,1,2,1})^\vee=(3,2)\quad\text{and}\quad ({1,2,1,1})^\vee=(2,3).$ More generally, we have
\begin{align}\label{eq:HdualDef2}
{\bfk}^\vee=(\underbrace{1,\ldots,1}_{k_1}+\underbrace{1,\ldots,1}_{k_2}+1,\ldots,1+\underbrace{1,\ldots,1}_{k_r}).
\end{align}
To save space, for any $i,j\in\N$ we put
\begin{align*}
&{\ora\bfk}_{\hskip-1pt i,j}:=
\left\{
  \begin{array}{ll}
    (k_i,\ldots,k_{j}), \quad \ & \hbox{if $i\le j\le r$;} \\
    \emptyset, & \hbox{if $i>j$,}
  \end{array}
\right.
 \quad &\ola\bfk_{\hskip-1pt i,j}:=
\left\{
  \begin{array}{ll}
     (k_{j},\ldots,k_i), \quad\ & \hbox{if $i\le j\le r$;} \\
     \emptyset, & \hbox{if $i>j$.}
  \end{array}
\right.
\end{align*}
Set $\ora\bfk_{\hskip-1pt i}=\ora\bfk_{\hskip-1pt 1,i}=(k_1,\ldots,k_i)$ and $\ola\bfk_{\hskip-1pt i}=\ola\bfk_{\hskip-1pt i,r}=(k_r,\ldots,k_i)$ for all $1\le i\le p$.

\begin{lem}\label{lem-MPL-1/2-Dual} For any composition $\bfk=(k_1,\ldots,k_r)\in \N^r$, we have
\begin{align}\label{Eq-MPL-1/2}
\int_0^1 \frac{\Li_{\bfk}(t)}{t\sqrt{1-t}}dt=2^{|\bfk|+1}t((1,k_r,k_{r-1},\ldots,k_1)^\vee).
\end{align}
\end{lem}
\begin{proof}
According to the definition of classical multiple polylogarithm function, we have the iterated integral expression
\begin{align}\label{Eq-MPL-ItIn}
\Li_{k_1,\ldots,k_r}(x)=\int_0^x \left(\frac{dt}{1-t}\right)\left(\frac{dt}{t}\right)^{k_r-1}\cdots
\left(\frac{dt}{1-t}\right)\left(\frac{dt}{t}\right)^{k_1-1}.
\end{align}
Hence
\begin{align*}
\int_0^1 \frac{\Li_{\bfk}(t)}{t\sqrt{1-t}}dt&=\int_0^1 \left(\frac{dt}{1-t}\right)\left(\frac{dt}{t}\right)^{k_r-1}\cdots
\left(\frac{dt}{1-t}\right)\left(\frac{dt}{t}\right)^{k_1-1} \frac{dt}{t\sqrt{1-t}} \\
&\overset{t\to 1-t^2}{=}2^{|\bfk|+1} \int_0^ 1\frac{dt}{1-t^2} \left(\frac{t\, dt}{1-t^2} \right)^{k_1-1}\frac{dt}{t}\cdots \left(\frac{t\, dt}{1-t^2} \right)^{k_r-1}\frac{dt}{t}\\
&=2^{|\bfk|+1}t((1,k_r,k_{r-1},\ldots,k_1)^\vee).
\end{align*}
This completes the proof of the lemma.
\end{proof}

\begin{lem}\label{lem:xn-1Li} \emph{(\cite[Thm. 2.1]{XuZhao2020b})}
Let $r,n\in \N$, ${\bfk}:=(k_1,\dotsc,k_r)\in \N^r$. Then
\begin{align}\label{Eq-x-n-MPL-IN}
 \int\limits_0^1 {x^{n-1}{\rm Li}_{\bfk}(x)dx}
 &= \frac{(-1)^{|\bfk|-r}}{n^{k_1}} \ze_{n}^{\star}(\ora{\bfk}_{\hskip-2pt 2,r},1) +\sum\limits_{j=0}^{k_{1}-2}(-1)^{j}\frac{\ze(k_1-j,\ora{\bfk}_{\hskip-1pt 2,r})}{n^{j+1}}\nonumber\\
&
 +\sum\limits_{l=1}^{r-1}(-1)^{\ora{\mid{\bfk}_{l}\mid}-l}\sum\limits_{j=0}^{k_{l+1}-2}(-1)^{j}
\frac{\ze_{n}^{\star}(\ora{\bfk}_{\hskip-1pt 2,l},j+1)}{n^{k_{1}}}\ze(k_{l+1}-j,\ora{\bfk}_{\hskip-1pt l+2,r}) .
\end{align}
\end{lem}

\begin{thm}\label{thm-mtss-mhs-cb}
Let $\bfk=(k_1,\ldots,k_r)\in \N^r$. If $p\in \N$ then we have
\begin{align}
&\frac{(-1)^p}{2^pp!}\int_0^1\frac{\Li_{\bfk}(x)\log^p(1-x)}{x\sqrt{1-x}}dx=\sum_{n=1}^\infty \frac{4^n}{\binn} \frac{t_n^\star(1_p)\ze_{n-1}(k_2,\ldots,k_r)}{n^{k_1+1} }\label{ChenKW}\\
&
\aligned
=&(-1)^{|\bfk|-r}\sum_{n=1}^\infty \frac1{4^n}\binn \frac{t_n(1_p) \ze_{n}^{\star}(\ora{\bfk}_{\hskip-2pt 2,r},1)}{n^{k_1}}
+\sum\limits_{j=0}^{k_{1}-2}(-1)^{j}\ze(k_1-j,\ora{\bfk}_{\hskip-2pt 2,r})\sum_{n=1}^\infty \frac1{4^n}\binn \frac{t_n(1_p)}{n^{j+1}}
\\
&
+\sum\limits_{l=1}^{r-1}(-1)^{\ora{\mid{\bfk}_{l}\mid}-l}\sum\limits_{j=0}^{k_{l+1}-2}(-1)^{j}
\ze(k_{l+1}-j,\ora{\bfk}_{l+2,r})\sum_{n=1}^\infty \frac1{4^n}\binn \frac{t_n(1_p)\ze^\star_n(\ora{\bfk}_{2,l},j+1)}{n^{k_1}} \in \MtV.
\endaligned
\label{Eq-mtss-mhs-cb}
\end{align}
If $p=0$, then we have
\begin{align}
&\int_0^1\frac{\Li_{\bfk}(x)}{x\sqrt{1-x}}dx=\sum_{n=1}^\infty \frac{4^n}{\binn} \frac{\ze_{n-1}(k_2,\ldots,k_r)}{n^{k_1+1}}
=2^{|\bfk|+1} t((1,k_r,k_{r-1},\ldots,k_1)^\vee)\nonumber\\
&\aligned
=& \sum\limits_{l=1}^{r-1}(-1)^{\ora{\mid{\bf k}_{l}\mid}-l}\sum\limits_{j=0}^{k_{l+1}-2}(-1)^{j}
\ze(k_{l+1}-j,\ora{\bfk}_{l+2,r})\sum_{n=1}^\infty \frac{1}{4^n}\binn \frac{\ze^\star_n(\ora{\bfk}_{2,l},j+1)}{n^{k_1}}\\
&
+\ze(k_1+1,\ora{\bfk}_{2,r})
+\sum\limits_{j=0}^{k_{1}-2}(-1)^{j}\ze(k_1-j,\ora{\bfk}_{2,r})\sum_{n=1}^\infty
\frac{1}{4^n}\binn \frac{1}{n^{j+1}} .
\endaligned
\label{Eq-mtss-mhs-cb-k-0}
\end{align}
\end{thm}
\begin{proof}
From \eqref{Eq-n-k-Log-Sqrtx}, we obtain
\begin{align}\label{Eq-iteratintegral-mtss-cb}
\frac{4^n}{n\binn}t_n^\star(1_p)=\frac{(-1)^p}{2^pp!}\int_0^1 x^{n-1}\frac{\log^p(1-x)}{\sqrt{1-x}}dx.
\end{align}
According to the definition of classical multiple polylogarithm function, we have
\begin{align}\label{MPL-II-2}
\Li_{k_1,\ldots,k_r}(x)=\sum_{n=1}^\infty \frac{\ze_{n-1}(k_2,\ldots,k_r)}{n^{k_1}}x^n.
\end{align}
Multiplying \eqref{Eq-iteratintegral-mtss-cb} by $\frac{\ze_{n-1}(k_2,\ldots,k_r)}{n^{k_1}}$, summing up, and applying \eqref{Eq-GF-Log-Sqrt1-x} and \eqref{Eq-x-n-MPL-IN} we get
\begin{align*}
\sum_{n=1}^\infty \frac{4^n}{n^{k_1+1}\binn}t_n^\star(1_p)\ze_{n-1}(k_2,\ldots,k_r)
&\,=\frac{(-1)^p}{2^pp!}\int_0^1 \frac{\Li_{k_1,\ldots,k_r}(t)\log^p(1-t)}{t\sqrt{1-t}}dt\nonumber\\
&\,=\sum_{n=1}^\infty \frac{1}{4^n}\binn t_n(1_p)\int_0^1 x^{n-1}\Li_{k_1,\ldots,k_r}(x)dx
\end{align*}
by \eqref{Eq-GF-Sums}. Hence \eqref{Eq-mtss-mhs-cb} follows immediately from Lemma \ref{lem:xn-1Li}.
Next, according to the iterated integral shuffle relation, we know that
$$\Li_{k_1,\ldots,k_r}(t)\log^p(1-t)=(-1)^p \Li_{k_1,\ldots,k_r}(t) \Li_1(t)^p
$$
can be expressed in terms of a rational linear combination of $\Li_\bfm(t)$ with
weight $|\bfm|=|\bfk|+p$ and depth $r+p$. For example, we have
\begin{equation*}
\Li_{2,2}(t)\log(1-t)=-2\Li_{2,2,1}(t)-2\Li_{2,1,2}(t)-\Li_{1,2,2}(t).
\end{equation*}
Thus we see that the left-hand side of \eqref{ChenKW} lies in $\MtV$ by Lemma \ref{lem-MPL-1/2-Dual}.

Similarly, applying \eqref{Eq-GF-sqrt1-x} and the fact
\begin{equation*}
\int_0^1 \frac{x^{n-1}}{\sqrt{1-x}} dx=B\Big(\frac12,n\Big)=\frac{4^n}{n\binn}  \qquad(\text{by }\eqref{BetaHalf})
\end{equation*}
we can also deduce  \eqref{Eq-mtss-mhs-cb-k-0} by an argument similar to that used in the proof of \eqref{Eq-mtss-mhs-cb}.
\end{proof}

\begin{re} It should be pointed out that Chen \cite[Eq. (5.2)]{ChenKW19} already obtained the explicit formula \eqref{ChenKW}.
\end{re}

\begin{cor}\label{COR-mhs-mths-cb-mtvs1}
Let $p\in\N_0$, $m\in \N$ and $\bfk\in \N^r$. Then we have
\begin{align*}
\sum_{n=1}^\infty \frac{4^n}{\binn} \frac{t^\star_n(1_{p})\ze_n(\bfk)}{n^{m+1}}\in \CMZV_{|\bfk|+m+p+1}^2.
\end{align*}
\end{cor}
\begin{proof} This follows from Thm.~\ref{thm-mtss-mhs-cb} since all MtVs lie in  $\CMZV^2$ and
\begin{equation*}
\ze_n(\bfk)=\frac1{n^{k_1}}\ze_{n-1}(k_2,\dots,k_r)+\ze_{n-1}(\bfk).
\end{equation*}
\end{proof}

\begin{cor}\label{COR-mhs-mths-cb-mtvs} For integers $m,r\in \N$ and $p\in \N_0$, we have
\begin{align}\label{Eq-mhs-mths-cb-mtvs}
\sum_{n=1}^\infty \frac1{4^n}\binn \frac{t_n(1_{p})\ze^\star_n(1_r)}{n^m}\in \CMZV_{p+r+m}^2.
\end{align}
\end{cor}
\begin{proof} Letting $\bfk=(m,1_{r-1})$ in Thm.~\ref{thm-mtss-mhs-cb} yields
\begin{align*}
\sum_{n=1}^\infty\frac{4^n}{\binn} \frac{\ze_{n-1}(1_{r-1})t_n^\star(1_p)}{n^{m+1}}&=\sum_{j=1}^{m-1} (-1)^{j-1} \ze(m+1-j,1_{r-1})\sum_{n=1}^\infty \frac1{4^n}{\binn} \frac{t_n(1_p)}{n^j}\\
&\quad+(-1)^{m-1}\sum_{n=1}^\infty \frac1{4^n}{\binn} \frac{t_n(1_{p})\ze^\star_n(1_r)}{n^m}+\delta_{0,p}\ze(m+1,1_{r-1}),
\end{align*}
where $\delta_{0,0}=1$ and $\delta_{0,p}=0$ for $p>0$. In particular, if letting $m=1$ and $p\in \N$ in the above formula, we obtain Thm.~\ref{thm-Apery-MtV-Mhs-CB} by replacing $r$ by $m$.

Applying \eqref{Eq-MtHS-CB} and noting that the \eqref{Eq-MtHS-CB} also holds for $p=0$ (see \cite[Thm. 3.4]{WX2021}), we obtain the desired description.
\end{proof}

\begin{exa} Setting $m=1,p=0$ yields
\[\sum_{n=1}^\infty  \frac1{4^n}\binn \frac{\ze^\star_n(1_r)}{n^2}=\ze(r+1,1)-2^{r+2}t(r+1,1)+2\ze(r+1)\log 2.\]
Setting $r=m=2,p=1$ we get
\begin{align*}
&\sum_{n=1}^\infty \frac{4^n}{\binn} \frac{\ze_{n-1}(1)t_n^\star(1)}{n^{3}}=16(3t(4,1)+t(3,2))=45\ze(4)\log 2-31\ze(5),\\
&\sum_{n=1}^\infty \frac1{4^n}\binn \frac{t_n(1)\ze^\star_n(1,1)}{n^2}=\frac3{2}\ze(2)\ze(3)+31\ze(5)-45\ze(4)\log 2.
\end{align*}
\end{exa}

\begin{cor} For integers $m,r\in \N$ and $p\in \N_0$, we have
\begin{align}\label{Eq-mhs-mths-cb-mtvs-2-2}
\sum_{n=1}^\infty \frac1{4^n}\binn \frac{t_n(1_{p})\ze^\star_n(2_{r-1},1)}{n^m}\in \CMZV_{p+m+2r-1}^2,
\end{align}
where $2_p$ means the string of $p$ repetitions of $2$'s.
\end{cor}
\begin{proof}
Similarly to the proof of Cor.~\ref{COR-mhs-mths-cb-mtvs}, setting $\bfk=(m,2_{r-1})$ in Thm.~\ref{thm-mtss-mhs-cb}, we can quickly arrive at \eqref{Eq-mhs-mths-cb-mtvs-2-2}.
\end{proof}

\begin{exa} Taking $p\le 1$ and $r=2$, we have
\begin{align*}
&\sum_{n=1}^\infty \frac1{4^n}\binn \frac{\ze^\star_n(2,1)}{n^2} =\frac{75}{8}\ze(5)-4\ze(4)\log 2-3\ze(2)\ze(3),\\
&\sum_{n=1}^\infty \frac1{4^n}\binn \frac{t_n(1)\ze^\star_n(2,1)}{n^2} =\frac{1055}{32}\ze(6)-\frac{69}{4}\ze^2(3)+32\ze(\bar5,1)-62\ze(5)\log 2\\
&\quad\quad\quad\quad\quad\quad\quad\quad\quad\quad\quad+28\ze(2)\ze(3)\log 2-16\ze(2)\ze(\bar3,1).
\end{align*}
\end{exa}

\section{General MSV-MtV product variant of Ap\'{e}ry-type series, II}\label{sec-SpecialMZSV-MtVproductII}
In this section, we will use a different approach to derive more general formulas of
some MZSV-MtV product variant of Ap\'{e}ry-type series treated in the last section. The key idea is
to use the shuffle regularization of CMZVs developed in Racinet's thesis (see \cite{Racinet2002}).
For an exposition in English, the reader can consult the book by the second author \cite{Zhao2016}.

\begin{lem}\label{thm-ItI1} \emph{(cf.\ \cite[Thm. 2.1]{XuZhao2021a})}
For any composition $\bfk=(k_1,\ldots,k_r)\in\N^r$, $n\in\N$, and $|x|\le 1$, we have
\begin{align}\label{FII1}
\int_0^x  \frac{t^n\, dt}{1-t} \left(\frac{dt}{t}\right)^{k_1-1}\cdots \frac{dt}{1-t}\left(\frac{dt}{t}\right)^{k_r-1}
=(-1)^r\ze^\star_n(\bfk;x)-\sum\limits_{j=1}^{r}(-1)^{j} \ze^\star_n(\ora\bfk_{\hskip-2pt j-1})\Li_{\ola\bfk_{\hskip-1pt j}}(x),
\end{align}
where
\begin{equation*}
\ze^\star_n(\bfk;x):=\sum\limits_{n\geq n_1\geq\dotsm \geq n_r\geq 1}\frac{x^{n_r}}{n_1^{k_1}\cdots n_r^{k_r}}.
\end{equation*}
\end{lem}

\begin{lem}\label{lem-mzv-mzsv-cb-mmvs}
Let $p,m\in\N_0$, $r\in \N$ and $\bfk=(k_1,\ldots,k_r)\in \N^r$. If $|x|<1$ then
\begin{align}\label{eq-mzv-mzsv-cb-mmvs}
& (-1)^r \sum_{n=1}^\infty \frac{1}{4^n}\binn\frac{\ze^\star_n(\bfk;x)t_n(1_p)}{n^{m+1}}
-\sum_{j=1}^r (-1)^j \Li_{\ola{\bfk}_{\hskip-2pt j}}(x)\sum_{n=1}^\infty \frac{1}{4^n}\binn
\frac{\ze^\star_n(\ora{\bfk}_{\hskip-2pt j-1})t_n(1_p)}{n^{m+1}}\nonumber\\
&= 2^r\int_{\sqrt{1-x}}^1 \left( \frac{2tdt}{1-t^2}\right)^{k_r-1}\frac{dt}{t}\cdots\left( \frac{2tdt}{1-t^2}\right)^{k_1-1}\frac{dt}{t} \left( \frac{2tdt}{1-t^2}\right)^{m}\chi_p\left(\frac{dt}{t}\right)^p,
\end{align}
where $\chi_0=\frac{2dt}{1+t}$ and $\chi_p=\frac{2dt}{1-t^2}$ for $p\ge 1$.
\end{lem}
\begin{proof}
If $p>0$ then from \eqref{Eq-GF-Log-Sqrt1-x} we have
\begin{align*}
2^p\sum_{n=1}^\infty \frac{1}{4^n}\binn \frac{t_n(1_p)}{n^{m+1}}x^n=\int_0^x \left(\frac{dt}{1-t}\right)^p\frac{dt}{t\sqrt{1-t}}\left(\frac{dt}{t}\right)^{m}.
\end{align*}
Multiplying \eqref{FII1} by $\frac{2^p}{4^n}\binn\frac{t_n(1_p)}{n^{m+1}}$  we see that
\begin{align*}
&2^p(-1)^r \sum_{n=1}^\infty  \frac{1}{4^n}\binn \frac{\ze^\star_n(\bfm;x)t_n(1_p)}{n^{m+1}}
-2^p \sum_{j=1}^r (-1)^j \Li_{\ola\bfk_{\hskip-1pt j}}(x)\sum_{n=1}^\infty  \frac{1}{4^n}\binn
\frac{\ze^\star_n(\ora\bfk_{\hskip-2pt j-1}))t_n(1_p)}{n^{m+1}} \\
&=\int_0^x \left(\frac{dt}{1-t}\right)^p\frac{dt}{t\sqrt{1-t}}\left(\frac{dt}{t}\right)^{m}\frac{dt}{1-t}\left(\frac{dt}{t}\right)^{k_1-1}\cdots \frac{dt}{1-t}\left(\frac{dt}{t}\right)^{k_r-1}  \\
&\overset{t\to 1-t^2}{=} 2^{r+p}\int_{\sqrt{1-x}}^1 \left( \frac{2tdt}{1-t^2}\right)^{k_r-1}\frac{dt}{t}\cdots\left( \frac{2tdt}{1-t^2}\right)^{k_1-1}\frac{dt}{t} \left( \frac{2tdt}{1-t^2}\right)^{m}\frac{2dt}{1-t^2}\left(\frac{dt}{t}\right)^p.
\end{align*}
If $p=0$ then we need to use \eqref{Eq-GF-Log-Sqrt1-x0} and see that the
\begin{align*}
&2^p(-1)^r \sum_{n=1}^\infty  \frac{1}{4^n}\binn \frac{\ze^\star_n(\bfm;x) }{n^{m+1}}
-2^p \sum_{j=1}^r (-1)^j \Li_{\ola\bfk_{\hskip-1pt j}}(x)\sum_{n=1}^\infty  \frac{1}{4^n}\binn
\frac{\ze^\star_n(\ora\bfk_{\hskip-2pt j-1})) }{n^{m+1}} \\
&=\int_0^x \left(\frac1{\sqrt{1-t}}-1\right)\frac{dt}{t}\left(\frac{dt}{t}\right)^{m}\frac{dt}{1-t}\left(\frac{dt}{t}\right)^{k_1-1}\cdots \frac{dt}{1-t}\left(\frac{dt}{t}\right)^{k_r-1}  \\
&\overset{t\to 1-t^2}{=} 2^{r+p}\int_{\sqrt{1-x}}^1 \left( \frac{2tdt}{1-t^2}\right)^{k_r-1}\frac{dt}{t}\cdots\left( \frac{2tdt}{1-t^2}\right)^{k_1-1}\frac{dt}{t} \left( \frac{2tdt}{1-t^2}\right)^{m}\frac{2dt}{1+t} .
\end{align*}
This completes the proof of the lemma.
\end{proof}

Put $\ta=\frac{dt}{t}$ and for $N\in\N$ denote by $\Gamma_N$ the group of $N$th roots of unity.
Set $\tx_\xi=\frac{dt}{\xi-t}$ for any $\xi\in\Gamma_N$.

\begin{lem} \label{lem:RegularizationLemma}
Let $r\in\N$, $\ell\in\N_0$ and
$\ga_1, \ldots, \ga_r\in  \{\emph{\text{$\tx$}}_\xi: \xi\in\Gamma_N\}\cup\{\emph{\text{$\ta$}}\}$ with $\ga_1\ne \emph{\text{$\ta$}}$ and $\ga_r\ne\emph{\text{$\tx$}}_1$, namely, $\bfga=\ga_1 \dots \ga_r$ is admissible.
Suppose $\eps\in(0,1)$ and the function $\tau(\eps)=O(\eps)$
satisfies $\log \tau(\eps)=A+\gl \log \eps+O(\eps)$
for some positive $\gl\in\Q$, $A=0$ or $A\in\CMZV^N_1$. Then we have
\begin{equation*}
\int_{\tau(\eps)}^1 \emph{\text{$\ta$}}^\ell \bfga  = P_\ell(\log \eps)+O(\eps\log^{l+r-1} \eps)
\end{equation*}
for some polynomial $P_\ell(T)\in \CMZV^N[T]$. Moreover, the coefficient of $T^{\ell-j}$ has weight $j+r$ for all $0\le j\le \ell$.
\end{lem}
\begin{proof} We proceed by induction on the total weight $w=r+\ell$. Let $\sha$ denote the shuffle
product of words used when multiplying iterated integrals first studied by K.T. Chen \cite{KTChen1977}.

The case $w=1$ is obvious. In general, if $\ell=0$ then this is trivial since
$\bfga$ is admissible and the value $\int_0^1\bfga$ is finite.
If $\ell\ge 1$ then we see that
\begin{align} \label{equ:aell1}
  \int_{\tau(\eps)}^1 \ta^\ell \bfga- \int_{\tau(\eps)}^1 \ta^\ell\int_{\tau(\eps)}^1\bfga
=&\, \int_{\tau(\eps)}^1 ( \ta^\ell \bfga  - \ta^\ell \sha\bfga)
=Q_{l-1}[\log \eps]
\end{align}
where by induction $Q_{l-1}[T]\in \CMZV^N[T]$ with its coefficient of $T^{\ell-j}$
having weight $j+r$ for all $1\le j\le \ell$. But $\log \tau(\eps)=A+\gl \log \eps+O(\eps)$ and therefore
\begin{equation}\label{equ:aell2}
\int_{\tau(\eps)}^1 \ta^\ell=\frac{(-1)^\ell}{\ell!} \log^\ell\big(\tau(\eps)\big)
=\frac{(-\gl)^\ell}{\ell!} \big( A+\gl\log \eps \big)^l +O(\eps\log^{l-1} \eps).
\end{equation}
On the other hand,
\begin{align*}
  \int_0^1\bfga- \int_{\tau(\eps)}^1 \bfga
=  \int_0^{\tau(\eps)} \bfga +\sum_{j=1}^{r-1} \int_0^{\tau(\eps)} \ga_1 \dots \ga_j
\int_{\tau(\eps)}^1 \ga_{j+1} \dots \ga_r.
\end{align*}
For all $j\le r$, since $\ga_1\ne \ta$ we may assume $\ga_1 \dots \ga_j=\tx_{\xi_1} \ta^{s_1-1}\dots\tx_{\xi_d} \ta^{s_d-1}$
for some $(s_1,\dots,s_d)\in\N^d$. Then $\int_0^{\tau(\eps)} \ga_1 \dots \ga_j=Li_{s_1,\dots,s_d}(\tau(\eps)/\xi_1,\dots)=O(\eps)$.
Furthermore, by induction we see that $\int_{\tau(\eps)}^1 \ga_{j+1} \dots \ga_r=O(\log^{r-1} \eps)$.
Hence
\begin{align} \label{equ:aell3}
\int_0^1 \bfga- \int_{\tau(\eps)}^1\bfga= O(\eps\log^{r-1} \eps).
\end{align}
Combining \eqref{equ:aell2} and  \eqref{equ:aell3} we have
\begin{align*}
\int_{\tau(\eps)}^1  \ta^\ell \int_{\tau(\eps)}^1  \bfga
= \int_{\tau(\eps)}^1  \ta^\ell \int_0^1 \bfga +O(\eps \log^{\ell+r-1}\eps)
= \frac{(-\gl)^\ell}{\ell!} \big( A+\gl\log \eps \big)^l \int_0^1 \bfga+O(\eps \log^{\ell+r-1}\eps),
\end{align*}
where $\int_0^1\bfga\in\CMZV_{r}^N$. Together with  \eqref{equ:aell1} this yields that
\begin{align*}
\int_{\tau(\eps)}^1   \ta^\ell \bfga
=&\, \frac{(-\gl)^\ell}{\ell!} \big( A+\gl\log \eps \big)^l \int_0^1 \bfga+Q_{\ell-1}(\log \eps)+O(\eps \log^{\ell+r-1}\eps)\\
=&\,  P_{\ell}(\log \eps)+O(\eps \log^{\ell+r-1}\eps).
\end{align*}
This completes the proof of the lemma.
\end{proof}

\begin{lem} \label{lem:Limit}
Suppose $m,p,r\in\N$ and $\bfk\in\N^r$. If a function $f(n,p)$ satisfies that there is a constant $C_p$ such that
$|f(n,p)|<C_p$ for all $n$ then we have
\begin{align*}
\lim_{x\to 1^-} \sum_{n=1}^\infty \frac{1}{4^n}\binn \frac{f(n,p) \big(\ze^\star_n(\bfk)-\ze^\star_n(\bfk;x)\big)}{n^{m}}=0,\\
\lim_{x\to 1^-} \sum_{n=1}^\infty \frac1{4^n}\binn\frac{f(n,p)\big(t^\star_{n}(\bfk)-t^\star_{n}(\bfk;x)\big)}{(2n+1)^m}
=0,\\
\lim_{x\to 1^-} \sum_{n=1}^\infty \frac{4^{n}f(n,p)\big(t^\star_n(\bfk)-t^\star_n(\bfk;x)\big)}{\binn n^{m+1}}
= 0,
\end{align*}
where
\begin{equation*}
t^\star_n(\bfk;x):=\sum\limits_{n\geq n_1\geq\dotsm \geq n_r\geq 1}\frac{x^{2n_r-1}}{(2n_1-1)^{k_1}\cdots (2n_r-1)^{k_r}}.
\end{equation*}
\end{lem}
\begin{proof} Let $x\in (1/2,1]$. Then
\begin{align*}
|\ze^\star_{n}(\bfk)-\ze^\star_{n}(\bfk;x)|
\le \sum_{n\geq n_1\geq n_2\geq \cdots\geq n_r>0}  (1-x^{n})\prod_{j=1}^r \frac{1}{n_{j}}
=(1-x^{n}) \ze^\star_{n}(1_r)
\ll  (1-x^{n}) \log^{r}(n)
\end{align*}
and
\begin{align*}
|t^\star_{n}(\bfk)-t^\star_{n}(\bfk;x)|
= &\,  \sum_{n\geq n_1\geq n_2\geq \cdots\geq n_r>0} (1-x^{2n_r-1})\prod_{j=1}^r \frac{1}{2n_{j}-1}\\
= &\, \sum_{2n\geq n_1\geq n_2\geq \cdots\geq n_r>0} (1-x^{n_r})\prod_{j=1}^r \frac{1-(-1)^{n_j}}{n_{j}} \\
\le &\, 2^r  \sum_{2n\geq n_1\geq n_2\geq \cdots\geq n_r>0} (1-x^{2n})\prod_{j=1}^r \frac{1}{n_{j}}  \ll  (1-x^{2n}) \log^{r}(n).
\end{align*}
Observing that
$
 \frac1{4^n}\binn \sim \frac1{\sqrt{\pi n}}
$
by Stirling's formula we see that for $m\in\N$ all the series
\begin{align*}
\sum_{n=1}^\infty \frac1{4^n}\binn \frac{t^\star_{n}(\bfk)-t^\star_{n}(\bfk;x)}{(2n+1)^{m}},  \quad
\sum_{n=1}^\infty  \frac{1}{4^n}\binn  \frac{\ze^\star_n(\bfk)-\ze^\star_n(\bfk;x)}{n^{m}}, \quad
\sum_{n=1}^\infty \frac{4^{n} \big(t^\star_n(\bfk)-t^\star_n(\bfk;x)\big)}{\binn n^{m+1}}
\end{align*}
converge absolutely and uniformly for $x\in(1/2,1]$, which yields the lemma immediately.
\end{proof}

\begin{thm} \label{thm-mzv-mzsv-cb-mmvs}
For any $p\in \N_0, r,m\in \N$ and $\bfk=(k_1,\ldots,k_r)\in \N^r$, we have
\begin{align}
\sum_{n=1}^\infty \frac{1}{4^n}\binn \frac{\ze^\star_n(\bfk)t_n(1_p)}{n^{m}}\in \CMZV_{|\bfk|+m+p}^2.
\end{align}
\end{thm}
\begin{proof}
Taking $N=2$ and $\tau(\eps)=\sqrt{\eps}$ in Lemma \ref{lem:RegularizationLemma} we see that
the regularized value of the iterated integral in \eqref{eq-mzv-mzsv-cb-mmvs}
\begin{equation*}
\int_{\sqrt{\eps}}^1 \left( \frac{2tdt}{1-t^2}\right)^{k_r-1}\frac{dt}{t}\cdots\left( \frac{2tdt}{1-t^2}\right)^{k_1-1}\frac{dt}{t} \left( \frac{2tdt}{1-t^2}\right)^{m-1}\chi_p\left(\frac{dt}{t}\right)^p
\end{equation*}
is a multiple mixed value which lies in $\CMZV_{|\bfk|+m+p}^2$. Further, for all $\bfm=(m_1,\dots,m_d)\in\N^d$
\begin{equation*}
\ze^\star_\sha(\bfm;x)=\sum_{\circ=\text{``,'' or ``+''}} \ze_\sha(m_1\circ \dots\circ m_d;x),
\end{equation*}
where, by \cite[Lemma 13.3.29,Remark 13.3.23]{Zhao2016} each $\ze_\sha(\bfs;x)$ term is regarded as
a shuffle regularization of $\ze(\bfs)$ (setting $\bfs=(s_1,\dots,s_l)\in\N^l$)
\begin{equation*}
\ze_\sha(\bfs;x)=\Li_\bfs(1-\eps)=\int_0^{1-\eps} \tx_1\ta^{s_l-1}\cdots \tx_1\ta^{s_1-1}=F(\log \eps)+O(\eps\log^l \eps)
\end{equation*}
for some polynomial $F[T]\in \CMZV^1[T]$ whose constant term has weight $|\bfs|=|\bfm|$.
The theorem follows immediately from Lemma \ref{lem-mzv-mzsv-cb-mmvs} by induction on $r$ after
taking constant terms of the regularization of \eqref{eq-mzv-mzsv-cb-mmvs} (i.e., taking $x=1-\eps\to 1^-$),
using Lemma \ref{lem:Limit}.
\end{proof}

\begin{thm}  \label{thm-1stPower}
Suppose $m,p\in\N$ and $\bfk=(k_1,\ldots,k_r)\in \N^r$ with $r\in\N_0$, where $\bfk=\emptyset$ and $|\bfk|=0$ if $r=0$.
Set $t_n^\star(\emptyset)=1$. Then we have
\begin{align}\label{1stPowerMtV}
\sum_{n=1}^\infty \frac{\binn}{4^n}\frac{t_n(2_{p-1})t_n^\star(\bfk)}{(2n+1)^m}\in i\CMZV^4_{2p+m+|\bfk|-2},\\
\label{1stPowerMZV}
\sum_{n=1}^\infty \frac{4^{n}\ze_{n-1}(2_{p-1})t^\star_n(\bfk)}{\binn n^{m+1}}\in \CMZV^4_{2p+m+|\bfk|-1}.
\end{align}
\end{thm}

\begin{proof}
By the proof of \cite[p.~262, Prop.~15]{B1985}, for $p\in \N$ we have
\begin{align} \label{Arcsin-Apery1}
&\frac{(\arcsin x)^{2p-1}}{(2p-1)!}=\sum_{n=p-1}^\infty \frac{\binn t_n(2_{p-1})}{4^n} \frac{x^{2n+1}}{2n+1}
\end{align}
and
\begin{align}\label{Arcsin-Apery2}
&\frac{(\arcsin x)^{2p}}{(2p)!}=\sum_{n=p-1}^\infty  \frac{4^{n+1-p}\ze_n(2_{p-1})}{(2n+1)\binn} \frac{x^{2n+2}}{2n+2}=\sum_{n=p}^\infty \frac{4^{n-p}\ze_{n-1}(2_{p-1})}{n^2\binn}x^{2n}.
\end{align}
Moreover, we have the following iterated integral
\begin{align*}
(\arcsin x)^p &=\left(\int_0^x\frac{dt}{\sqrt{1-t^2}}\right)^p=p!\int_{0}^x \left(\frac{dt}{\sqrt{1-t^2}}\right)^p.
\end{align*}
Let
\begin{equation}\label{defn-Fpm}
F_{p,m}(x):=\int_0^x \left(\frac{dt}{\sqrt{1-t^2}}\right)^p\left(\frac{dt}{t}\right)^{m-1}.
\end{equation}
Then  \eqref{Arcsin-Apery1} and \eqref{Arcsin-Apery2} yield
\begin{align}
&F_{2p-1,m}(x)=\sum_{n=p-1}^\infty \frac{\binn t_n(2_{p-1})}{4^n} \frac{x^{2n+1}}{(2n+1)^{m}}, 
&F_{2p,m}(x)=\sum_{n=p}^\infty \frac{4^{n-p}\ze_{n-1}(2_{p-1})}{2^{m-1} \binn n^{m+1}}x^{2n}.      \label{Eq-Defn-Fpm-2}
\end{align}

By \cite[Thm.~3.6]{XuZhao2021a}, for any $\bfk=(k_1,\ldots,k_r)\in\N^r$, $n\in\N$ and $|x|<1$, we have
\begin{align}\label{FIIt1}
&\int_0^x \frac{t^{2n}dt}{1-t^2} \left(\frac{dt}{t}\right)^{k_1-1}\frac{tdt}{1-t^2} \left(\frac{dt}{t}\right)^{k_2-1}\cdots \frac{tdt}{1-t^2}\left(\frac{dt}{t}\right)^{k_r-1}\nonumber\\
=&\,(2n-1)\int_0^x t^{2n-2}dt \frac{tdt}{1-t^2} \left(\frac{dt}{t}\right)^{k_1-1}\frac{tdt}{1-t^2} \left(\frac{dt}{t}\right)^{k_2-1}\cdots \frac{tdt}{1-t^2}\left(\frac{dt}{t}\right)^{k_r-1}\nonumber\\
=&\,\sum_{j=1}^r (-1)^{j-1} t^\star_n(\ora\bfk_{\hskip-2pt j-1})t(\ola\bfk_{\hskip-2pt j};x)+(-1)^r t^\star_n(\bfk;x).
\end{align}
Multiplying \eqref{FIIt1} by
$\frac{\binn t_n(2_{p-1})}{4^n(2n+1)^m}$ and $\frac{4^{n-p}\ze_{n-1}(2_{p-1})}{2^{m-1} \binn n^{m+1}}$, respectively, then summing up, we see that
\begin{align}\label{formula-tmhss-mtv-apery-1}
&\sum_{j=1}^{r} (-1)^{j-1} t(\overleftarrow{\bfk}_j;x)\sum_{n=1}^\infty \frac{\binn}{4^n}\frac{t_n(2_{p-1})t_n^\star(\ora{\bfk}_{j-1})}{(2n+1)^m}
+(-1)^r \sum_{n=1}^\infty \frac{\binn}{4^n}\frac{t_n(2_{p-1})t_n^\star(\bfk;x)}{(2n+1)^m} \nonumber\\
=&\,\int_0^x \frac{F_{2p-1,m}(t)dt}{t(1-t^2)} \left(\frac{dt}{t}\right)^{k_1-1} \frac{tdt}{1-t^2} \left(\frac{dt}{t}\right)^{k_2-1}\cdots \frac{tdt}{1-t^2} \left(\frac{dt}{t}\right)^{k_r-1}\nonumber\\
=&\,\int_0^x \left(\frac{dt}{\sqrt{1-t^2}}\right)^{2p-1}\left(\frac{dt}{t}\right)^{m-1}\frac{dt}{t(1-t^2)} \left(\frac{dt}{t}\right)^{k_1-1} \frac{tdt}{1-t^2} \left(\frac{dt}{t}\right)^{k_2-1}\cdots \frac{tdt}{1-t^2} \left(\frac{dt}{t}\right)^{k_r-1}
\end{align}
and by \eqref{defn-Fpm}
\begin{align}\label{formula-tmhss-mtv-apery-2}
&\sum_{j=1}^{r} (-1)^{j-1} t(\overleftarrow{\bfk}_j;x)\sum_{n=1}^\infty \frac{4^{n-p}\ze_{n-1}(2_{p-1})t^\star_n(\ora{\bfk}_{j-1})}{2^{m-1} \binn n^{m+1}}
+(-1)^r \sum_{n=1}^\infty \frac{4^{n-p}\ze_{n-1}(2_{p-1})t^\star_n(\bfk;x)}{2^{m-1} \binn n^{m+1}}\nonumber\\
=&\,\int_0^x \frac{F_{2p,m}(t)dt}{1-t^2} \left(\frac{dt}{t}\right)^{k_1-1} \frac{tdt}{1-t^2} \left(\frac{dt}{t}\right)^{k_2-1}\cdots \frac{tdt}{1-t^2} \left(\frac{dt}{t}\right)^{k_r-1}\nonumber\\
=&\,\int_0^x \left(\frac{dt}{\sqrt{1-t^2}}\right)^{2p}\left(\frac{dt}{t}\right)^{m-1}\frac{dt}{t(1-t^2)} \left(\frac{dt}{t}\right)^{k_1-1} \frac{tdt}{1-t^2} \left(\frac{dt}{t}\right)^{k_2-1}\cdots \frac{tdt}{1-t^2} \left(\frac{dt}{t}\right)^{k_r-1}.
\end{align}

Applying $t\to \frac{1-t^2}{1+t^2}$ to \eqref{formula-tmhss-mtv-apery-1} and \eqref{formula-tmhss-mtv-apery-2}, setting $\ty=\tx_{-i}+\tx_{i}-\tx_{-1}-\tx_{1}$ and $\tz=-\ta-\tx_{-i}-\tx_{i}$, we get
\begin{alignat}{4} \label{equ:changeVar1}
  \ta=\frac{dt}{t}&\,  \to -\left(\frac{2tdt}{1+t^2}+\frac{2tdt}{1-t^2} \right)=\ty, &
\qquad
\om_2=&\, \frac{dt}{1-t^2}\to -\frac{dt}{t}=-\ta, \\
 \frac{dt}{\sqrt{1-t^2}} &\,  \to -\frac{2dt}{1+t^2}=i(\tx_{-i}-\tx_{i}), &
\qquad\om_3=&\, \frac{tdt}{1-t^2}\to -\left(\frac{dt}{t}-\frac{2tdt}{1+t^2}\right)=\tz,   \label{equ:changeVar2}\\
\frac{dt}{t(1-t^2)}&\,=\frac{dt}{t}+\frac{tdt}{1-t^2}  \to \ty+\tz=-\ta-& \tx_{-1}-\tx_{1}.
\label{equ:changeVar3}
\end{alignat}
Therefore, setting
\begin{equation*}
\tau(\eps):=\sqrt{\frac{\eps}{2-\eps}}
\end{equation*}
we have
\begin{align}\label{formula-tmhss-mtv-apery-cmzv-1}
&\sum_{j=1}^{r} (-1)^{j-1} t(\overleftarrow{\bfk}_j;1-\eps)\sum_{n=1}^\infty \frac{\binn}{4^n}\frac{t_n(2_{p-1})t_n^\star(\ora{\bfk}_{j-1})}{(2n+1)^m}
+\sum_{n=1}^\infty \frac{\binn}{4^n}\frac{t_n(2_{p-1})t_n^\star(\bfk;1-\eps)}{(2n+1)^m}
\nonumber\\
&=(-1)^{|\bfk|+m} \int_{\tau(\eps)}^1 \ty^{k_r-1}\tz\cdots \ty^{k_2-1}\tz\ty^{k_1-1}(\ty+\tz)y^{m-1}(i(\tx_{-i}-\tx_i))^{2p-1}
\end{align}
which lies in $i\CMZV^4[T]$ after regularization by Lemma \ref{lem:RegularizationLemma} with $N=4$, and
\begin{align}\label{formula-tmhss-mtv-apery-cmzv-2}
&\sum_{j=1}^{r} (-1)^{j-1} t(\overleftarrow{\bfk}_j;1-\eps)\sum_{n=1}^\infty \frac{4^{n-p}\ze_{n-1}(2_{p-1})t^\star_n(\ora{\bfk}_{j-1})}{2^{m-1} \binn n^{m+1}}
+\sum_{n=1}^\infty \frac{4^{n-p}\ze_{n-1}(2_{p-1})t^\star_n(\bfk;1-\eps)}{2^{m-1} \binn n^{m+1}}
\nonumber\\
&=(-1)^{|\bfk|+m-1}\int_{\tau(\eps)}^1 \ty^{k_r-1}\tz\cdots \ty^{k_2-1}\tz\ty^{k_1-1}\ta\ty^{m-1}(i(\tx_{-i}-\tx_i))^{2p}
\end{align}
which lies in $\CMZV^4[T]$ after regularization by Lemma \ref{lem:RegularizationLemma} with $N=4$.
Here we need the fact that $\log 2 =-\Li_1(-1)\in\CMZV_1^4$.
Hence, by induction on $r$ using the constant terms of the regularization of \eqref{formula-tmhss-mtv-apery-cmzv-1} and \eqref{formula-tmhss-mtv-apery-cmzv-2}, and Lemma \ref{lem:Limit},
we finish the proof of the  theorem.
\end{proof}

\begin{re}
One may want to compare statement \eqref{1stPowerMZV} with \cite[Thm. 4.1]{Au2020} in which
the product is replaced by multiple harmonic sums. Furthermore, Au
showed that the level of the CMZVs can be lowered to 2 in that case.
\end{re}

\begin{cor} \label{cor-1stPower}
For any $p,m,r\in\N$ and $\bfk\in\N^r$ or $\bfk=\emptyset$, we have
\begin{align}\label{equ:1stPowerMtV}
\sum_{n=1}^\infty \frac1{4^n}\binn\frac{t_n(\bfk)}{(2n+1)^{m}} &\, \in  i\CMZV^{4}_{|\bfk|+m},\\
\label{equ:1stPowerMZV}
\sum_{n=1}^\infty \frac{4^{n}}{\binn}\frac{t_n(\bfk)}{n^{m+1}} &\, \in  \CMZV^{4}_{|\bfk|+m+1},
\end{align}
where $t_n(\emptyset)=1$ and $|\emptyset|=0$.
\end{cor}
\begin{proof} Setting $p=1$ in Thm.~\ref{thm-1stPower} we see that the corollary follows from the fact
that
\begin{equation}\label{equ:starNonStar}
t_n(\bfk)=\sum_{\circ=\text{``,'' or ``+''}} (-1)^{\sharp\{\circ=\text{``+''}\}} t_n^\star(k_1\circ \cdots\circ k_r)
\end{equation}
which, in turn, follows easily from
\begin{equation}\label{equ:starNonStar2}
t_n^\star(\bfk)=\sum_{\circ=\text{``,'' or ``+''}}   t_n (k_1\circ \cdots\circ k_r)
\end{equation}
by the Principle of Inclusion and Exclusion.
\end{proof}

\begin{exa} Explicitly, set $\bfk=(1)$ in \eqref{equ:1stPowerMtV} we have
\begin{align*}
&\,\sum_{n=1}^\infty \frac1{4^n}\binn  \frac{t_{n}(1)}{(2n+1)^m}\\
=&\,
i(-1)^{m}\int_0^1\left(
\ty^{m-1}(\tx_{-i}-\tx_i)\ta+\sum_{j=1}^{m-1}  \ty^{m-j} \ta \ty^{j-1} (\tx_{-i}-\tx_i)
- (\tx_1+\tx_{-1})\ty^{m-1} (\tx_{-i}-\tx_i)\right).
\end{align*}
Further setting $m=1,2$  we get
\begin{align*}
\sum_{n=1}^\infty \frac1{4^n}\binn \frac{t_n(1)}{2n+1} &\,=4\big(G+{\rm Im}\Li_{1,1}(-i,i))\big) \approx 1.088793045,\\
\sum_{n=1}^\infty \frac1{4^n}\binn \frac{t_n(1)}{(2n+1)^2} &\,=
i\int_0^1\left( \ty(\tx_{-i}-\tx_i)\ta+\ty\ta (\tx_{-i}-\tx_i)
- (\tx_1+\tx_{-1})\ty(\tx_{-i}-\tx_i)\right)  \\
&\, \approx 0.108729731954.
\end{align*}
In the first equation, we have used the identity \cite[p.\ 48, (2.44)]{Zhao2016}:
$$
\Li_{1,1}(x,y)=\Li_2\Bigl(\frac{xy-x}{1-x}\Bigr)-\Li_2\Bigl(\frac{x}{x-1}\Bigr)-\Li_2(xy).
$$
Setting $\bfk=(2)$, $p=1$ and $m=1,2$ in \eqref{1stPowerMtV} we get
\begin{align*}
\sum_{n=1}^\infty \frac1{4^n}\binn \frac{t_n(2)}{2n+1} &\,=
-it(2)\int_0^1 (\tx_{-i}-\tx_{i})-i\int_0^1 \ty(\ta+\tx_{-1}+\tx_{1})(\tx_{-i}-\tx_{i})\approx 0.6459640977,\\
\sum_{n=1}^\infty \frac1{4^n}\binn \frac{t_n(2)}{(2n+1)^2} &\,=
it(2)\int_0^1 \ty(\tx_{-i}-\tx_{i}) +i\int_0^1 \ty(\ta+\tx_{-1}+\tx_{1})\ty(\tx_{-i}-\tx_{i})\approx 0.0937132114.
\end{align*}
Setting $\bfk=(2)$, $p=1,2$ and $m=1$ in \eqref{1stPowerMZV} we get
\begin{align*}
\sum_{n=1}^\infty \frac{4^n}{\binn}\frac{t_n(2)}{n^2} &\,=-4
\Big( t(2)\int_0^1(\tx_{-i}-\tx_{i})^2 + \int_0^1 \ty \ta(\tx_{-i}-\tx_{i})^2\Big)\approx 5.4641926215,\\
\sum_{n=1}^\infty \frac{4^n}{\binn}\frac{\ze_{n-1}(2)t_n(2)}{n^2} &\,=2^4
\Big( t(2)\int_0^1(\tx_{-i}-\tx_{i})^4 + \int_0^1 \ty a(\tx_{-i}-\tx_{i})^4\Big)\approx  4.822651414.
\end{align*}
\end{exa}

\begin{thm} \label{thm-2ndPower}
For any $p,k,m\in\N$ we have
\begin{align}
\label{Eq-mtss-1}
\sum_{n=p-1}^\infty  \bigg[\frac1{4^n}\binn\bigg]^2  \frac{t_n(2_{p-1})}{(2n+1)^{m}} &\, \in \frac{i}\pi \CMZV^{4}_{2p+m-1},\\
\sum_{n=p}^\infty \bigg[\frac{4^n}\binn\bigg]^2 \frac{\ze_{n-1}(2_{p-1})t_n^\star(1_k)}{n^{m+2}} &\,  \in \CMZV^{4}_{2p+k+m}.
\label{Eq-mtss-1-mhs-2-cb}
\end{align}
\end{thm}
\begin{proof} The proof is similar to that of Thm.~\ref{thm-1stPower}. First, it is well-known that
\begin{equation*}
\frac1{4^n}\binn=\frac2\pi\int_0^{\pi/2}\sin^{2n}\theta\, d\theta=\frac2\pi\int_0^1 \frac{t^{2n}}{\sqrt{1-t^2}}dt.
\end{equation*}
Multiplying this by
$\frac{\binn t_n(2_{p-1})}{4^n(2n+1)^m}$ and summing up, we see that
\begin{align}\label{Eq-iteratintegral-F2p-1-m-sqrt}
\sum_{n=1}^\infty \bigg[\frac1{4^n}\binn\bigg]^2 \frac{t_n(2_{p-1})}{(2n+1)^{m}}=
\frac2\pi \int_0^1 \frac{F_{2p-1,m}(t)dt}{t\sqrt{1-t^2}}.
\end{align}
Second, letting $t\to t^2$ in the integral on the right-hand side of \eqref{Eq-iteratintegral-mtss-cb}, we get
\begin{align}\label{Eq-iteratintegral-mtss-cb-2}
2^k\frac{4^n}{n\binn}t_n^\star(1_k)=2\int_0^1 \left(\frac{2tdt}{1-t^2}\right)^k \frac{t^{2n-1}dt}{\sqrt{1-t^2}}.
\end{align}
Multiplying \eqref{Eq-iteratintegral-mtss-cb-2} by $\frac{4^{n-p}\ze_{n-1}(2_{p-1})}{2^{m-1} \binn n^{m+1}}$ and summing up, we have
\begin{align}\label{Eq-iteratintegral-F2p-m-sqrt}
\sum_{n=1}^\infty \frac{\ze_{n-1}(2_{p-1})t_n^\star(1_k)}{n^{m+2}{\binn}^2}4^{2n}=
2^{2p+m-k}\int_0^1 \left(\frac{2tdt}{1-t^2}\right)^k \frac{F_{2p,m}(t)dt}{t\sqrt{1-t^2}}.
\end{align}
Applying $t\to \frac{1-t^2}{1+t^2}$, using \eqref{equ:changeVar1}--\eqref{equ:changeVar3} and
\begin{equation}\label{equ:changeVar4}
\frac{dt}{t\sqrt{1-t^2}} \to -\frac{2dt}{1-t^2}=\tx_{-1}-\tx_{1}
\end{equation}
we get
\begin{align*}
\text{RHS of  \eqref{Eq-iteratintegral-F2p-1-m-sqrt}} &\,
=  \frac{2i  (-1)^{p+m}}{\pi} \int_0^1(\tx_{-1}-\tx_{1}) \ty^{m-1} (\tx_{-i}-\tx_{i})^{2p-1},\\
\text{RHS of  \eqref{Eq-iteratintegral-F2p-m-sqrt}}
 &\,
 =(-1)^{p+m+k} 2^{2p+m}\int_0^1 (\tx_{-1}-\tx_{1})\Big(\ty^{m-1} (\tx_{-i}-\tx_{i})^{2p}\sha  (\tx_0-\tx_{i}-\tx_{-i})^k\Big),
\end{align*}
both of which lie in $\CMZV^4$. This concludes the proof of the theorem.
\end{proof}

\begin{exa}
Setting $p=1,k=1,m=1$, we get
\begin{align*}
\sum_{n=1}^\infty &\,\bigg[\frac{4^n}\binn\bigg]^2 \frac{t_n(1)}{n^3}
=-8\int_0^1 (\tx_{-1}-\tx_{1}) \Big( (\tx_{-i}-\tx_{i})^2\sha (\tx_0-\tx_{i}-\tx_{-i})\Big)\\
=&\, 8\sum_{\eta_1,\eta_2=\pm i,\eta_3=\pm1}\frac{-\eta_1}{\eta_2\eta_3}
\left(\Li_{2,1,1}\Big(\frac1{\eta_1},\frac{\eta_1}{\eta_2},\frac{\eta_2}{\eta_3}\Big)
+\Li_{1,2,1}\Big(\frac1{\eta_1},\frac{\eta_1}{\eta_2},\frac{\eta_2}{\eta_3}\Big)
+\Li_{1,1,2}\Big(\frac1{\eta_1},\frac{\eta_1}{\eta_2},\frac{\eta_2}{\eta_3}\Big)\right)\\
-&\, 8\sum_{\eta_1,\eta_2,\eta_3=\pm i,\eta_4=\pm1} \eta_4\left(\frac{\eta_2}{\eta_3}+\frac{\eta_1}{\eta_3}+\frac{\eta_1}{\eta_2}\right)
 \Li_{1,1,1,1}\Big(\frac1{\eta_1},\frac{\eta_1}{\eta_2},\frac{\eta_2}{\eta_3},\frac{\eta_3}{\eta_4}\Big)
\approx 7.7112698415,\\
\sum_{n=0}^\infty &\, \bigg[\frac1{4^n}\binn\bigg]^2  \frac{1}{2n+1}  =
 \frac{2 i}{\pi} \int_0^1(\tx_{-1}-\tx_{1}) (\tx_{-i}-\tx_{i}) \\
&\, \hskip3cm =\frac{2i}{\pi}\big(\Li_{1,1}(i,i)+\Li_{1,1}(-i,i)-\Li_{1,1}(i,-i)-\Li_{1,1}(-i,-i)\big)=\frac{4G}{\pi}
\end{align*}
which is consistent with \eqref{4GI}. Now taking $p=2$ or 3, $k=1$ and $m=1$, we have
\begin{align*}
\sum_{n=2}^\infty \bigg[\frac{4^n}\binn\bigg]^2 \frac{\ze_{n-1}(2)t_n(1)}{n^3}
=&\,2^5\int_0^1 (\tx_{-1}-\tx_{1}) \Big(  (\tx_{-i}-\tx_{i})^4\sha(\tx_0-\tx_{i}-\tx_{-i})\Big)\approx 4.8416943704,\\
\sum_{n=2}^\infty \bigg[\frac{4^n}\binn\bigg]^2 \frac{\ze_{n-1}(2,2)t_n(1)}{n^3}
=&\,-2^7\int_0^1 (\tx_{-1}-\tx_{1}) \Big(  (\tx_{-i}-\tx_{i})^6\sha(\tx_0-\tx_{i}-\tx_{-i})\Big)\approx 1.3105783945,\\
\sum_{n=1}^\infty  \bigg[\frac1{4^n}\binn\bigg]^2  \frac{t_n(2)}{2n+1} =&\,
 \frac{-2 i}{\pi} \int_0^1(\tx_{-1}-\tx_{1}) (\tx_{-i}-\tx_{i})^3  \approx 0.179386942,\\
\sum_{n=1}^\infty  \bigg[\frac1{4^n}\binn\bigg]^2  \frac{t_n(2,2)}{2n+1} =&\,
 \frac{2 i}{\pi} \int_0^1(\tx_{-1}-\tx_{1}) (\tx_{-i}-\tx_{i})^5  \approx  0.0139754925.
\end{align*}
If we let  $p=1$ or 2, $k=1$ and $m=2$, then we have
\begin{align}
\sum_{n=1}^\infty \bigg[\frac{4^n}\binn\bigg]^2 \frac{t_n(1)}{n^4}
=&\,2^4\int_0^1 (\tx_{-1}-\tx_{1}) \Big( \ty(\tx_{-i}-\tx_{i})^2\sha (\tx_0-\tx_{i}-\tx_{-i})\Big)\approx 5.0319188594,\notag\\
\sum_{n=1}^\infty \bigg[\frac{4^n}\binn\bigg]^2 \frac{\ze_{n-1}(2)t_n(1)}{n^4}
=&\,-2^6\int_0^1 (\tx_{-1}-\tx_{1}) \Big(\ty(\tx_{-i}-\tx_{i})^4\sha (\tx_0-\tx_{i}-\tx_{-i})\Big)\approx 1.1896632248,\notag\\
\sum_{n=0}^\infty  \bigg[\frac1{4^n}\binn\bigg]^2  \frac{1}{(2n+1)^2} =&\,
-\frac{2 i}{\pi} \int_0^1(\tx_{-1}-\tx_{1})\ty(\tx_{-i}-\tx_{i})\approx  1.037947765, \label{4GIm=2Num}\\
\sum_{n=0}^\infty  \bigg[\frac1{4^n}\binn\bigg]^2  \frac{t_n(2)}{(2n+1)^2} =&\,
\frac{2 i}{\pi} \int_0^1(\tx_{-1}-\tx_{1})\ty(\tx_{-i}-\tx_{i})^3 \approx  0.0393547288. \notag
\end{align}
We see that \eqref{4GIm=2Num} is consistent with \eqref{4GIm=2}.
\end{exa}

In fact, we can slightly generalize \eqref{Eq-mtss-1} to involve a factor of multiple harmonic star sum, which
requires the following lemma.

\begin{lem}\label{lem:star1s}
For all $m,n\in\N$, we have
\begin{equation*}
\sum\limits_{i=1}^{m}\ze_{n}^{\star}(1_{m-i}) \ze_{n}(i) =m  \ze_{n}^{\star}(1_{m}).
\end{equation*}
\end{lem}
\begin{proof} This is an immediate consequence of Lemma \ref{lem-S-1} if we take $x_k=1/k$ for all $k$. Alternatively,
we can prove it directly by induction. The lemma is clearly true if $m=1$. Suppose $m\ge 2$.
Fix $n$ and let $\star$ denote the stuffle product of the word algebra representing the
multiple harmonic star sums $\ze_n(\cdots)$ (see \cite{HoffmanIh2012}). For examples, $\ze_n^\star(1,4,1,3)$
is represented by the word $ba^3b^2a^2b$ and $\ze_{n}^{\star}(1_{m})$ by $b^m$. Then we have
\begin{align*}
 \sum_{i=1}^m b^{m-i} \star a^{i-1}b &\, =a^{m-1}b+\sum_{i=1}^{m-1}
 \big(a^{i-1}b^{m-i+1}+b(a^{i-1}b \star b^{m-i-1})-a^i b^{m-i}\big)\\
&\,=b^m+\sum_{i=2}^m a^{i-1}b^{m-i+1}+(m-1) b^m -\sum_{i=1}^{m-1} a^i b^{m-i}
\end{align*}
by induction. The lemma follows immediately by a simple index shifting.
\end{proof}

\begin{pro} \label{pro:BetaDer}
Let $b_0=1$ and for all $j\ge 1$ define recursively
$$b_j=-\frac1j \sum_{l=1}^j 2^l b_{j-l}\ze(\bar l).$$
Then $b_j$ is in the weight $j$ piece of $\Q[\log 2, \ze(2),\ze(3),\ze(4),\ldots]$ (assuming $\log 2$ has weight one) for all $j\ge 1$.
Moreover, we have
\begin{equation}
\frac{\partial^{k}B(a,b)}{\partial b^{k}}\bigg|_{a=n+1/2,b=1/2}
=\frac{(-1)^k k! \pi\binn}{4^{n}}\sum\limits_{j=0}^{k} b_j \ze_{n}^{\star}(1_{k-j}). \label{equ:BetaPartialb}
\end{equation}
\end{pro}
\begin{proof}
By induction it is easy to see that $b_j\in \Q[\log 2, \ze(2),\ze(3),\ze(4),\ldots]$ and has weight $j$.

We now prove \eqref{equ:BetaPartialb} by induction on $k$. The case $k=0$ is clear.
Now suppose \eqref{equ:BetaPartialb} is true for
all partial derivatives of order up to $k-1$. Observe that
\begin{align}
\frac{\partial^{k}B(a,b)}{\partial b^{k}}=\sum\limits_{i=0}^{k-1}\binom{k-1}{i}
\frac{\partial^{i}B(a,b)}{\partial b^{i}}\cdot\left[\psi^{(k-i-1)}(b)-\psi^{(k-i-1)}(a+b)\right]
\end{align}
and
\begin{align*}
\psi^{(k)}(1/2)-\psi^{(k)}(n+1)
\,&=(-1)^{k+1}k!\sum\limits_{m=0}^{\infty}
\left\{\frac{1}{(1/2+m)^{k+1}}-\frac{1}{(n+1+m)^{k+1}}\right\} \\
&=(-1)^{k+1}k!\left\{2^{k+1}t(k+1)-\ze(k+1)+\ze_{n}(k+1)\right\} \\
&=(-1)^{k+1}k!\left\{-2^{k+1}\ze(\overline{k+1})+\ze_{n}(k+1)\right\}.
\end{align*}
Hence, setting $a=n+1/2,b=1/2$ we get
\begin{align*}
\frac{\partial^{k}B(a,b)}{\partial b^{k}}\bigg|_{a=n+1/2,\atop b=1/2}
&=(-1)^{k}(k-1)!\sum\limits_{i=0}^{k-1}\frac{(-1)^{i}}{i!}\cdot\frac{\partial^{i}B(a,b)}{\partial b^{i}}\bigg|_{a=n+1/2,\atop b=1/2}
\left\{\ze_{n}(k-i)-2^{k-i}\ze(\overline{k-i})\right\}.    \label{equ:betaRecur}
\end{align*}
This yields that
\begin{align*}
\frac{(-1)^k 4^{n}}{\pi\binn(k-1)!}\frac{\partial^{k}B(a,b)}{\partial b^{k}}\bigg|_{a=n+1/2,\atop b=1/2}
&= \sum\limits_{i=0}^{k-1}
\sum\limits_{j=0}^{i}  b_j \ze_{n}^{\star}(1_{i-j})
\left\{\ze_{n}(k-i)-2^{k-i}\ze(\overline{k-i})\right\}\\
&=\sum\limits_{j=0}^{k-1}  \sum\limits_{i=1}^{k-j} b_j \ze_{n}^{\star}(1_{k-i-j}) \ze_{n}(i)
-\sum\limits_{i=0}^{k-1} \sum\limits_{j=0}^{i}  b_{i-j} \ze_{n}^{\star}(1_{j})  2^{k-i}\ze(\overline{k-i})\\
&=\sum\limits_{j=0}^{k-1}  b_j  \sum\limits_{i=1}^{k-j}\ze_{n}^{\star}(1_{k-i-j}) \ze_{n}(i)
-\sum\limits_{j=1}^{k} \sum\limits_{l=1}^{j}  b_{j-l} \ze_{n}^{\star}(1_{k-j})2^{l}\ze(\overline{l})\\
&=\sum\limits_{j=0}^{k-1}  b_j  (k-j) \ze_{n}^{\star}(1_{k-j})+\sum\limits_{j=1}^{k} j b_{j} \ze_{n}^{\star}(1_{k-j})
\end{align*}
by Lemma \ref{lem:star1s} and inductive assumption, respectively. The proposition follows readily.
\end{proof}

\begin{thm} For any $p,m\in\N$ and $k\in\N_0$,
\begin{equation}\label{formua-binom-t-mhs}
\sum\limits_{n=p-1}^{\infty}\bigg[\frac\binn{4^n}\bigg]^2
\frac{t_{n}(2_{p-1})\ze_{n}^{\star}(1_{k})}{(2n+1)^{m}}\in \frac{i}{\pi} \CMZV^{4}_{2p+k+m-1}.
\end{equation}
\end{thm}

\begin{proof}
Note that for any integers $n,k\geq 0$,
\begin{align*}
\int\limits_0^1{\left(\frac{tdt}{1-t^{2}}\right)^{k}\frac{t^{2n}}{\sqrt{1-t^{2}}}dt}
&=\frac{(-1)^{k}}{2^{k}k! }\int\limits_0^1{t^{2n}\frac{\log^{k}(1-t^{2})}{\sqrt{1-t^{2}}}dt}\nonumber\\
&=\frac{(-1)^{k}}{2^{k+1}k! }\int\limits_0^1 {x^{n-1/2}\frac{\log^{k}(1-x)}{\sqrt{1-x}}dx}\\
&=\frac{(-1)^{k}}{2^{k+1}k! }\frac{\partial^{k}B(a,b)}{\partial b^{k}}\bigg|_{a=n+1/2,\atop b=1/2}
=\frac{\pi}{2^{k+1}}\frac{\binn}{4^{n}}\sum\limits_{j=0}^{k}b_j\ze_{n}^{\star}(1_{k-j})
\end{align*}
by Prop.~\ref{pro:BetaDer}. Multiplying this by
$\frac{\binn t_n(2_{p-1})}{4^n(2n+1)^m}$ and summing up, we see that
\begin{align}
&\sum\limits_{j=0}^{k}  b_j \sum\limits_{n=p-1}^{\infty}\bigg[\frac\binn{4^n}\bigg]^2
\frac{t_{n}(2_{p-1})\ze_{n}^{\star}(1_{k-j})}{(2n+1)^{m}}\nonumber\\
&=\frac{2^{k+1}}{\pi}\int\limits_0^1{\left(\frac{tdt}{1-t^{2}}\right)^{k}\sum\limits_{n=p-1}^{\infty}
\frac{\binn t_{n}(2_{p-1})}{4^{n}(2n+1)^{m}}\frac{t^{2n}}{\sqrt{1-t^{2}}}dt}\nonumber\\
&=\frac{2^{k+1}}{\pi}\int\limits_0^1{\left(\frac{tdt}{1-t^{2}}\right)^{k}\frac{F_{2p-1,m}(t)}{t\sqrt{1-t^{2}}}dt}\in \frac{i}{\pi} \CMZV^{4}_{2p+k+m-1}.
\end{align}
The theorem follows easily from an induction on $k$.
\end{proof}

\begin{thm} \label{thm-MtVsm222}
For any $p,m\in\N_0$ we have
\begin{align*}
t(m+2,2_p)=\int_0^1 \left(\frac{dt}{\sqrt{1-t^2}}\right)^{2p+1} \left(\frac{dt}{t}\right)^{m}\frac{dt}{\sqrt{1-t^2}}.
\end{align*}
\end{thm}
\begin{proof} Replacing $n$ by $n+1$ in \eqref{Eq-iteratintegral-mtss-cb-2}, multiplying by $\frac{\binn t_n(2_p)}{4^n(2n+1)^{m+1}}$ and summing up, we obtain
\begin{align*}
\int_0^1 \left(\frac{2tdt}{1-t^2}\right)^k \frac{F_{2p+1,m+1}(t)dt}{\sqrt{1-t^2}}=2^{k} \sum_{n=1}^\infty \frac{t_{n+1}^\star(1_k)t_n(2_p)}{(2n+1)^{m+2}}.
\end{align*}
The theorem follows immediately if we set $k=0$.
\end{proof}

The following result can be found at the end of the proof of \cite[Thm.~2.8]{Chavan-KL2021}. We can now prove it
as an corollary of Thm.~\ref{thm-MtVsm222}.
\begin{cor}
For any $p\in\N_0$ we have
\begin{align*}
t(3,2_p)
=\frac{1}{(2p+1)!}\int_0^1  \frac{(\arcsin t)^{2p+1} \arccos t\, dt}{t}
\end{align*}
\end{cor}
\begin{proof} Note that
\begin{equation*}
\int_0^t \left(\frac{dx}{\sqrt{1-x^2}}\right)^{2p+1} =\frac{(\arcsin t)^{2p+1}}{(2p+1)!}.
\end{equation*}
Taking $m=1$ in the Thm.~\ref{thm-MtVsm222} we get
\begin{align*}
t(3,2_p)=\frac{1}{(2p+1)!}\int_0^1  \int_0^t \frac{(\arcsin x)^{2p+1} dx}{x}  \frac{dt}{\sqrt{1-t^2}}.
\end{align*}
The corollary follows easily from integration by parts once and the fact that $ \arccos t=\pi/2- \arcsin t$.
\end{proof}

\begin{thm} \label{thm-anInvSqare}
For integers $p,k,m\in\N_0$ with $m\geq 3$, we have
\begin{align}
\sum_{n\ge 0}  \bigg[\frac{4^n}{\binn}\bigg]^2 \frac{t_n^\star(1_k)\ze_n(2_{p})}{(2n+1)^{m}}\in \CMZV_{m+2p+k}^4.
\end{align}
\end{thm}
\begin{proof}
For all $k\ge 1$ we have
\begin{align*}
t_n^\star(1_k)= t_{n+1}^\star(1_k)-\frac{t_{n+1}^\star(1_{k-1})}{2n+1}.
\end{align*}
Thus Thm.~\ref{thm-anInvSqare} follows immediately from the next theorem.
\end{proof}

\begin{thm} \label{thm-anInvSq2}
For integers $s\geq 0$ and $p\geq 1$, we have
\begin{align*}
\sum_{n\ge 0}  \bigg[\frac{4^n}{\binn}\bigg]^2 \frac{t_{n+1}^\star(1_k)\ze_n(2_{p-1})}{(2n+1)^{s+3}}\in \CMZV_{s+2p+k+1}^4.
\end{align*}
\end{thm}
\begin{proof}
Differentiating \eqref{Arcsin-Apery2} and then dividing by $x$  we get
\begin{align*}
\sum_{n=p-1}^\infty  \frac{4^{n+1-p}\ze_n(2_{p-1})}{(2n+1)\binn} x^{2n}
=\frac{1}{x\sqrt{1-x^2}}\frac{(\arcsin x)^{2p-1}}{(2p-1)!}
=\frac{1}{x\sqrt{1-x^2}} \int_0^x \Big(\frac{dt}{\sqrt{1-t^2}}\Big)^{2p-1}.
\end{align*}
Repeatedly integrating and then dividing by $x$ exactly $s$ times leads to
\begin{align}\label{Formula-zeta-even-x}
\sum_{n\ge 0}  \frac{4^{n+1-p}\ze_n(2_{p-1})}{(2n+1)^{s+2}\binn} x^{2n}
=\frac{1}{x}\int_0^x  \Big(\frac{dt}{\sqrt{1-t^2}}\Big)^{2p-1}\frac{dt}{t\sqrt{1-t^2}}  \Big(\frac{dt}{t} \Big)^s
\end{align}
where we have used the convention that $\ze_n(2_{p-1})=0$ if $n<p-1$. Since
 \begin{align*}
\binom{2n+2}{n+1}^{-1} \frac1{n+1}=\frac{1}{2}\binn^{-1}\frac1{2n+1}
\end{align*}
we see that
\begin{align*}
 &\ \sum_{n\ge 0}  \bigg[\frac{4^n}{\binn}\bigg]^2 \frac{t_{n+1}^\star(1_k)\ze_n(2_{p-1})}{(2n+1)^{s+3}}
=\sum_{n\ge 0}  \frac{4^n}{\binn} \frac{\ze_n(2_{p-1})}{(2n+1)^{s+2}}\cdot \frac12\frac{4^{n+1}}{(n+1)\binom{2n+2}{n+1}}t_{n+1}^\star(1_k)\\
 &= \sum_{n\ge 0}  \frac{4^n}{\binn} \frac{\ze_n(2_{p-1})}{(2n+1)^{s+2}} \int_0^1 \left(\int_0^x\Big(\frac{2tdt}{1-t^2}\Big)^k\right) \frac{x^{2n+1}}{\sqrt{1-x^2}} dx \qquad \qquad \qquad \quad (\text{by \eqref{Eq-iteratintegral-mtss-cb-2}})\\
 & =4^{p-1} \int_0^1 \left(\int_0^x\Big(\frac{2tdt}{1-t^2}\Big)^k\right)  \left(\int_0^x  \Big(\frac{dt}{\sqrt{1-t^2}}\Big)^{2p-1} \frac{dt}{t\sqrt{1-t^2}}\Big(\frac{dt}{t}\Big)^s \right)  \frac{dx}{\sqrt{1-x^2}} \quad (\text{by \eqref{Formula-zeta-even-x}})\\
 & =4^{p-1} \int_0^1 \left[\Big(\frac{2tdt}{1-t^2}\Big)^k \sha \left(\Big(\frac{dt}{\sqrt{1-t^2}}\Big)^{2p-1} \frac{dt}{t\sqrt{1-t^2}}\Big(\frac{dt}{t}\Big)^s \right)\right]\frac{dt}{\sqrt{1-t^2}}
\end{align*}
By the change of variables $t\to \frac{1-t^2}{1+t^2}$ using \eqref{equ:changeVar1}-\eqref{equ:changeVar3} and \eqref{equ:changeVar4}
we see easily that this iterated integral lies in $\CMZV_{s+2p+k+1}^4$.
This completes the proof of the theorem.
\end{proof}

\section{Concluding remarks and questions}
In this paper, we have proved there are a few variant families of Ap\'ery type series such that each one can be expressed in terms of the real and imaginary part of CMZVs of the same weight and level (divided by $\pi$ sometimes). In general, however, as manifested by Example \ref{eg:mixedWt} not every Ap\'ery type series has this property.
Hence, we would like to conclude our paper with the following questions.

\begin{qu} Let $a_n= \frac1{4^n}{\binn} $. Is it true that
for all $m\in\N$, $p\in\N_{\ge2}$, $q\in\N_{\ge3}$, and all compositions of positive integers $\bfk$ and $\bfl$ (including the cases $\bfk=\emptyset$ or $\bfl=\emptyset$),
\begin{alignat*}{4}
{\rm (i)}& \ \ \sum_{n=1}^\infty a_n \frac{\ze_n(\bfk)t_n(\bfl)}{n^m}  \in \CMZV_{|\bfk|+|\bfl|+m}^4,& \quad
{\rm (ii)}& \ \  \sum_{n=0}^\infty  a_n^2 \frac{\ze_n(\bfk)t_n(\bfl)}{n^m}  \in  \frac{1}\pi \CMZV^{4}_{|\bfk|+|\bfl|+m+1},\\
{\rm (iii)}& \ \ \sum_{n=1}^\infty a_n^{-1} \frac{\ze_n(\bfk)t_n(\bfl)}{n^{p}}  \in \CMZV_{|\bfk|+|\bfl|+p}^4,& \quad
{\rm (iv)}& \ \ \sum_{n=1}^\infty a_n^{-2} \frac{\ze_n(\bfk)t_n(\bfl)}{n^{q}}  \in \CMZV^{4}_{|\bfk|+|\bfl|+q},\\
{\rm (v)}& \ \ \sum_{n=0}^\infty  a_n \frac{\ze_n(\bfk)t_n(\bfl)}{(2n+1)^{m}} \in  i\CMZV^{4}_{|\bfk|+|\bfl|+m},& \quad
{\rm (vi)}& \ \  \sum_{n=0}^\infty  a_n^2  \frac{\ze_n(\bfk)t_n(\bfl)}{(2n+1)^{m}} \in \frac{i}\pi \CMZV^{4}_{|\bfk|+|\bfl|+m+1},\\
{\rm (vii)}& \ \ \sum_{n=0}^\infty  a_n^{-1} \frac{\ze_n(\bfk)t_n(\bfl)}{(2n+1)^{p}}  \in  i\CMZV^{4}_{|\bfk|+|\bfl|+p},& \quad
{\rm (viii)}& \ \ \sum_{n=0}^\infty  a_n^{-2} \frac{\ze_n(\bfk)t_n(\bfl)}{(2n+1)^{q}} \in \CMZV^{4}_{|\bfk|+|\bfl|+q}?
\end{alignat*}
Can any of these be improved to $\CMZV^2$?
\end{qu}
In our recent papers \cite{XuZhao2022a,XuZhao2022b} we use a completely different approach to answer (i)-(vii) affirmatively.
Note that Thm.~\ref{thm-1stPower} (resp. Thm.~\ref{thm-2ndPower})
provides the affirmative answer for some special cases of (iii) and (v) (resp. (iv) and (vi)).
Moreover, Thm.~\ref{thm-anInvSqare} confirm (viii) in some special cases.
Furthermore, when $\bfl=\emptyset$ (i) and (ii) follow from \cite[Thm. 4.1]{Au2020} and \cite[Thm. 4.14]{Au2020}, respectively.

To conclude, we remark that many results in this paper can be regarded as special cases of the above eight families of Ap\'ery type series
even though they use the star version, because we have the well-known relations \eqref{equ:starNonStar}
and \eqref{equ:starNonStar2} between the star and non-star versions. These two relations hold for the multiple harmonic sums $\ze_n$, too.

\medskip

\noindent{\bf Acknowledgement.}  The first author is supported by the National Natural Science Foundation of China (Grant No. 12101008), the Natural Science Foundation of Anhui Province (Grant No. 2108085QA01) and the University Natural Science Research Project of Anhui Province (Grant No. KJ2020A0057). Jianqiang Zhao is supported by the Jacobs Prize from The Bishop's School.

\medskip

\noindent{\bf Disclosure statement.} The authors report there are no competing interests to declare.

\end{document}